\documentclass[12pt]{amsart}
\usepackage{xcolor}
\usepackage[draft]{todonotes} 
\usepackage[english]{babel}
\usepackage{geometry}
\usepackage[T1]{fontenc}
\usepackage[latin1]{inputenc}
\usepackage{amsmath,amssymb,mathrsfs,amsthm, tikz-cd,mathrsfs}
\usepackage{soul}
\usepackage{epsfig}
\usepackage{hyperref}

\usepackage{graphicx}
\usepackage{xypic}
\usepackage{enumitem}

\usepackage{datetime}
\usepackage{xcolor}
\usepackage{fancyhdr}

\usepackage{hyperref} 
\hypersetup{
    colorlinks=true,
    linkcolor=blue,
    urlcolor=red,
    pdftitle={ },
    }

\DeclareMathOperator*\Rprod{%
\mathchoice
  {\ooalign{$\displaystyle\prod$\cr\hidewidth$\displaystyle\coprod$\hidewidth\cr}}%
  {\ooalign{$\textstyle\prod$\cr\hidewidth$\textstyle\coprod$\hidewidth\cr}}%
  {\ooalign{$\scriptstyle\prod$\cr\hidewidth$\scriptstyle\coprod$\hidewidth\cr}}%
  {\ooalign{$\scriptscriptstyle\prod$\cr\hidewidth$%
  \scriptscriptstyle\coprod$\hidewidth\cr}}}

\newcommand\Vbullet{\raisebox{-2.1pt}{\kern-0.4pt\vbox{\baselineskip4pt\lineskiplimit0pt%
\hbox{$\bullet$}\hbox{$\bullet$}}}}

\newtheorem{theorem}{Theorem}[section]
\newtheorem*{theorem*}{Theorem}
\newtheorem{proposition}[theorem]{Proposition}

\newtheorem{lemma}[theorem]{Lemma}

\newtheorem{definition}[theorem]{Definition}
\newtheorem{example}[theorem]{Example}
\newtheorem{remark}[theorem]{Remark}

\newcommand{\C}{\mathbb{C}}

\newcommand{\Z}{\mathbb{Z}}

\newcommand{\vf}{\varphi}

\newcommand{\Q}{\mathbb{Q}}

\newcommand{\pt}{\partial}

\begin{document}
\title{Regularized Products over arithmetic Schemes}

\author{Mounir Hajli}
\address{School of Science, Westlake University\\
Hangzhou, China 310000}
\email{hajli@westlake.edu.cn}


            \date{\today, \currenttime}

\begin{abstract}

In this paper, we study the closed points of arithmetic schemes. We accomplish this by showing that the product of the cardinals of residue fields of closed points in an arithmetic scheme can be regularized. This regularization yields a new arithmetic invariant attached to the scheme. We compute it explicitly in several cases and find that it is always a transcendental number. This result provides a proof that the set of closed points is infinite. Our main tool is a regularization technique, which generalizes the zeta-regularization method introduced by Mu\~noz and P\'erez.
\end{abstract}
\maketitle
\tableofcontents

\section{Introduction}

Zeta regularization of infinite products is an important technique used in number theory, geometry, and physics~\cite{Deninger-gamma, Deninger-L, RaySinger, Hawking1}. 
For instance, this technique is used to define regularized products of sequences of arithmetic nature, such as:
\[
\Rprod_{n=1}^\infty n = \sqrt{2\pi}, \quad
\Rprod_{n=1}^\infty (n^2+1) = e^{\pi} - e^{-\pi}, \quad
\Rprod_{n \in \mathrm{Sf}} n = 2\pi, \quad
\Rprod_{n=1}^\infty q^n = q^{-\frac{1}{12}} \quad (q > 1), \quad \ldots
\]
where $\mathrm{Sf}$ denotes the set of squarefree positive integers. 
For further examples of regularized products of  {arithmetic}  nature, see~\cite{Allouche-odious, Allouche-Cohen, Allouche-arith-sequences, Kurokawa1, Kurokawa-2005}.

Prime numbers and their distribution lie at the heart of analytic number theory and are intimately connected to the Riemann zeta function. For instance, the \emph{Prime Number Theorem} rests on the fact that the Riemann zeta function does not vanish on the line $\mathrm{Re}(s) = 1$, i.e., on $1 + i\mathbb{R}$; see, for instance, \cite{Titchmarsh}.
Soul\'e~\cite{Soule} raised the natural question: Can the product over all prime integers be similarly regularized? 
Mu\~{n}oz and Marco solved this problem in~\cite{Marco_prime}. 
The main breakthrough is a new zeta regularization technique that generalizes the classical definition of zeta regularization of an infinite product. 
As a consequence, they proved that the infinite product over all prime integers can be regularized using this new technique, thus answering the question of  Soul\'e. 
They showed that this regularized product equals $4\pi^2$. 
Noticing that this number is irrational, they deduced a new analytic proof of the fact that the set of prime numbers is infinite. Recently, Krapivsky and Luck~\cite{Luck} provided heuristic computations  \emph{\`a la Euler} of the regularized infinite products of Gauss and Eisenstein integers and primes.

\bigskip

The main objects of interest of this article are the \emph{zeta functions of  {arithmetic}  nature}. 
These  functions are objects of profound fascination. 
One of the central problems in number theory and algebraic geometry over the past decades has been to understand the subtle relationship between the special values of zeta functions of varieties defined over number fields and certain geometric invariants. In other words, it is known (or conjectured) that the values of certain $L$-functions are intimately related to  invariants in number theory and  geometry,  see for example \cite{Geisser-2004, Morin-2018, Lichtenbaum-2005}. \\

 Let $X$ be a regular scheme, projective and flat over $\operatorname{Spec}(\mathbb{Z})$, and equidimensional of relative dimension $d$. To such a scheme $X$, one may associate a zeta function $\zeta(X, s)$, which encodes essential  {arithmetic}  information about $X$. This is the \emph{Hasse--Weil zeta function} of $X$, defined by the Euler product
\begin{equation}\label{eq1}
\zeta(X, s) = \prod_{x} \left(1 - N(x)^{-s}\right)^{-1},
\end{equation}
where the product runs over all closed points $x \in X$, and $N(x)$ denotes the cardinal of the residue field at $x$.  It is known that $\zeta(X,s)$ is a holomorphic function for $\mathrm{Re}(s)>\dim X$. Let $L(X, s)$ denote the completed zeta function, namely $
L(X, s) := \zeta(X, s)  \Gamma(X, s).
$
It is conjectured that $L(X, s)$ satisfies a functional equation of the form
\[
L(X, s) = \epsilon(X)  A(X)^{-s}  L(X, d+1-s),
\]
where $\epsilon(X)$ and $A(X) $ are constants
 known respectively as the \emph{$\epsilon$-constant} and the \emph{conductor} of $X$. The product \eqref{eq1} expands into a Dirichlet series $\sum_n a_n n^{-s}$ with integer coefficients, since the number of closed points $x  \in X$ having a given cardinal  $N(x)$ is finite.\\

If $X = \operatorname{Spec}(A)$, where $A$ is a Dedekind domain, then the Hasse--Weil zeta function takes the form
\begin{equation}
\zeta(\operatorname{Spec}(A), s) = \sum_{\mathfrak{a}} N\mathfrak{a}^{-s}, \quad \text{for } \operatorname{Re}(s) > 1,
\end{equation}
where the sum runs over all nonzero ideals $\mathfrak{a}$ of $A$, and $N\mathfrak{a} = \#(A/\mathfrak{a})$ denotes the norm of $\mathfrak{a}$.

In the case where $A = \mathcal{O}_K$ is the ring of integers of a number field or a function field $K$,  we abbreviate $
\zeta_K(s) := \zeta(\operatorname{Spec}(\mathcal{O}_K), s).$ 
This is the classical Dedekind zeta function of $K$.  It has a simple pole at $s=1$, and 
\[\zeta_K(s)\sim (s-1)^{-1} \frac{2^{r_1} (2\pi)^{r_2}h_K R_K }{w_K |D_K|^{1/2}}\quad  \text{as}\ s\rightarrow 1\]
Here, $r_1$ (resp.\ $r_2$) is the number of real (resp.\ complex) places of $K$, $D_K$ is the discriminant of the field $k$, $h_K$ is its class number, $R_K$ is its regulator, and $w_K$ is the number of roots of unity contained in $K$. The functional equation of $\zeta_K(s)$ allows translating this theorem into a statement at the point $s = 0$. Indeed, taking into account well-known properties of the $\Gamma$ function, one easily deduces   that
the Taylor expansion of $\zeta_K(s)$ at $s = 0$ begins with the term
\[
-\frac{h_KR_K}{w_K} s^{r_1 + r_2 - 1}.
\]
A similar result holds for
\[
\zeta_{K,S}(s) := \prod_{\mathfrak{p} \notin S} (1 - N\mathfrak{p}^{-s})^{-1},
\]
where $S$ is a finite set of places of $K$ containing the set $S_\infty$ of infinite places of $K$, and the product is taken over the finite places of $K$ not belonging to $S$. Namely, one can show that
\begin{equation}\label{cK}
\zeta_{K,S}(s) = -\frac{h_S R_S}{w_S}  s^{\operatorname{card}(S) - 1} + O(s^{\operatorname{card}(S)}),
\quad \text{as } s \to 0,
\end{equation}
where $h_S$, $R_S$, and $w_S$ are constants defined analogously to the classical case: $h_S$ is the $S$-class number, $R_S$ is the $S$-regulator, and $w_S$ is the number of roots of unity in the ring of $S$-integers. In fact, $\zeta_{K,S}(s)$ is the Dedekind zeta function attached to the ring $
\mathcal{O}_{K,S} = \bigcap_{\mathfrak{p} \notin S} \mathcal{O}_{\mathfrak{p}}.$ \\

Now let $F$ be a finite Galois extension of $K$, with Galois group $G$. Let $V$ be finite complex dimensional representation of $G$. Let   $\chi: G \to \mathbb{C} $ be 
the character of the  representation $G \to \mathrm{GL}(V)$.  With the finite set $S$ of places of $K$ fixed, the \emph{Artin $L$-function} (relative to $S$) attached to $\chi$ is given by the following Euler product:
\begin{equation}
L(s, \chi) = L_S(s, \chi) = \prod_{\mathfrak{p} \notin S} \det(1 - \sigma_{\mathfrak{P}} N\mathfrak{p}^{-s} \mid V^{I_{\mathfrak{
P}}})^{-1},
\end{equation}
where $\mathfrak{P}$ denotes an arbitrary place of $F$ above $\mathfrak{p}$, and $\sigma_{\mathfrak{p}} \in G_{\mathfrak{p}}/I_{\mathfrak{p}}$ is the Frobenius substitution of the residue extension $\mathfrak{P}/\mathfrak{p}$. 
Near $s = 0$, we write
\begin{equation}\label{cL}
L(s, \chi) = c(\chi)  s^{r(\chi)} + O(s^{r(\chi)+1}),
\end{equation}
where $r(\chi)$ denotes the multiplicity of $0$. It is known that $r(\chi)$ can determined explicitely in terms of the representation $V$.  
Stark 
\cite{Stark-1, Stark-2, Stark-3, Stark-4, Stark-Hilbert} formulated several conjectures about $c(\chi)$, the leading term of Artin $L$-functions of Galois extensions of algebraic number fields.   In general the conjectures state that the leading term should be the ``\emph{Stark regulator}''. \\

Another interesting class of zeta functions arises from schemes defined over finite fields. Let $X$ be a smooth projective scheme over the finite field $\mathbb{F}_q$, where $q$ is a prime power. It is known that
\[
\zeta(X, s) = \frac{P_1(X, q^{-s}) \cdots P_{2\dim X - 1}(X, q^{-s})}{P_0(X, q^{-s}) P_2(X, q^{-s}) \cdots P_{2\dim X}(X, q^{-s})},
\]
where each $P_i(X, t)$ is a polynomial in $\mathbb{Z}[t]$, for $i = 0, \dots, 2\dim X$. For further details, see Section~\ref{schemefinite}.\\

One can observe that the above zeta functions admit a Laurent expansion at $s = 0$ of the form
\[
\zeta(s) = c_r s^{r} + c_{r+1} s^{r+1} + O(s^{r+2}) \quad \text{as } s \to 0,
\]
 where $r$ is an integer and $c_r$ is a leading coefficient known  (at least conjecturally) in many cases (see \eqref{cK}, \eqref{cL}, Section~\ref{schemefinite}, and \cite{DixExposes, Milne1, Tate-zeta1, Tate-abelian, Tate-zeta2, Schneider-zeta}). In contrast, significantly less is known about the next coefficient, $c_{r+1}$.\\

In this paper, we investigate the problem of regularizing infinite products over primes, as encoded by certain zeta-like series associated to  {arithmetic}  schemes. Specifically, we consider the series
\[
\mathcal{P}_X(s) := \sum_{x} \frac{1}{N(x)^s},
\]
where the sum runs over all closed points $x \in X$, and $N(x)$ denotes the cardinal  of the residue field at $x$ in an  {arithmetic}  scheme $X$.\\

We shall investigate in  depth two key classes of such series $\mathcal{P}_X$:

\smallskip
\noindent(1) The Dedekind-type zeta series attached to a global field $K$:
\[
\mathcal{P}_K(s) := \sum_{\mathfrak{p}} \frac{1}{(N\mathfrak{p})^s},
\]
where the sum ranges over all nonzero prime ideals $\mathfrak{p}$ of the ring of integers $\mathcal{O}_K$, and $K$ is either a number field (i.e., a finite extension of $\mathbb{Q}$) or a global function field (i.e., a finite extension of $\mathrm{Fr}(C)$, where $C$ is a smooth projective curve over a finite field $k$).

\smallskip

\noindent {(2)} The Dirichlet-type series for  {arithmetic}  progressions:
\[
\mathcal{P}_{m,a}(s) := \sum_{\substack{p\ \text{prime} \\ p \equiv a\, \mathrm{mod}\, m}} \frac{1}{p^s},
\]
where $m$ is a positive integer and $a \in (\mathbb{Z}/m\mathbb{Z})^\times$, and the sum is over rational primes $p$ congruent to $a$ modulo $m$.  To $\mathcal{P}_{m,a}$, one can associate two distinct classical zeta functions: The Dirichlet series  
\[
\sum_{p \text{ prime}} \frac{\chi(p)}{p^s},
\]  
where $\chi$ is a Dirichlet character, and the partial zeta function  
\[
\zeta_{m,a}(s) = \prod_{\substack{p \text{ prime} \\ p \equiv a\, \mathrm{mod}\, m}} \frac{1}{1 - p^{-s}},
\label{eq:zeta_da}
\]  
which restricts the Euler product to primes congruent to $a \bmod m$.   This latter function has been studied extensively in connection with the distribution of primes in  {arithmetic}  progressions,  see ~\cite{Landau-1,Landau-2}.  Moree~\cite{Moree-2}, and Languasco, Moree~\cite{Moree-1}, investigated related Euler constants arising from such primes, see also Fouvry et al.~\cite{Waldschmidt-1}.

The above zeta functions are canonical tools for analyzing the distribution of ideals and prime ideals in algebraic number fields.  For instance, it is known that the set $A$ of all prime ideals of $\mathcal O_K$ is regular, and its Dirichlet density equals $1$, which means that 
\[
\sum_{\mathfrak{p} \in A} \frac{1}{N(\mathfrak{p})^s} =  \log \frac{1}{s - 1} + g(s)
\]
where $g$ is a regular function in the closed half-plane $\sigma>1$.\\

We conclude this introduction by outlining how the scope of the theory developed in this paper may be extended to encompass more natural zeta functions arising from geometric problems. For instance, consider the series
\[
\sum_{P \in \mathrm{Prim}(\Gamma)} \frac{1}{N(P)^s},
\]
where $\mathrm{Prim}(\Gamma)$ denotes the set of hyperbolic primitive conjugacy classes of a discrete co-compact subgroup $\Gamma \subset \mathrm{SL}_2(\mathbb{R})$. This series is naturally associated with Ruelle-type zeta functions, such as:
\[
\zeta_\Gamma(s) = \prod_{P \in \mathrm{Prim}(\Gamma)} \left(1 - N(P)^{-s}\right)^{-1}
\]
and
\[
Z_\Gamma(s) = \prod_{n=0}^\infty \prod_{P \in \mathrm{Prim}(\Gamma)} \left(1 - N(P)^{-s-n}\right)\tag{Higher Selberg zeta function}
\]
See \cite{Hejhal, Selberg1956, Kurokawa-1989, Kurokawa-2004}. These are related by  
\[
\zeta_\Gamma(s) = \frac{Z_\Gamma(s+1)}{Z_\Gamma(s)}.
\]
We aim, in future work, to compute the regularized product
\[
\underset{P \in \mathrm{Prim}(\Gamma)}{\Rprod} N(P),
\]
in terms of special values of $\zeta_\Gamma(s)$.


\subsection*{Organization of the paper}
\
\begin{itemize}
\item 
In Section~\ref{SR}, we introduce a generalization of the notion of super-regularization of infinite products due to Mu\~{n}oz and P\`erez. 
Note the remark in~\cite[p.~73]{Marco_prime} regarding a possible extension of this notion. Our assumption allows us to extend the definition of the regularized product given by Kurokawa and Wakayama~\cite{Kurokawa1, Kurokawa-2005}.
\item 
In Section~\ref{ZPNF}, we consider the following Dirichlet series \[
\mathcal{P}_K(s)=\sum_{\frak p } \frac{1}{(N \frak p)^s}
\]
where the sum runs over  nonzero prime ideals of a number field $K$ and $N\frak p$ is the norm of a prime ideal $\frak p$. We apply the machinery developed in Section \ref{SR} 
that the infinite product of norms of nonzero prime ideals of $K$ can regularized. This regularized product is denoted by
\[
\Rprod_{\mathfrak{p} \in \mathrm{Spec}(\mathcal{O}_K) \setminus \{0\}} N\mathfrak{p}.
\]
We show in Theorem \ref{NF} that 
the regularized product over all nonzero prime ideals $\mathfrak{p} \in \mathrm{Spec}(\mathcal{O}_K) \setminus \{0\}$ satisfies:
\[
\Rprod_{\mathfrak{p} \in \mathrm{Spec}(\mathcal{O}_K) \setminus \{0\}} N\mathfrak{p} = \exp\left( -\frac{2 w_K}{(r_1 + r_2) h_K R_K}  \zeta_K^{(r_1 + r_2)}(0) \right),
\]
where
$w_K$ is the number of roots of unity in $K$, $r_1$, $r_2$ are the numbers of real and complex embeddings of $K$,
  $h_K$ is the class number,
  $R_K$ is the regulator, and 
  $\zeta_K^{(n)}(0)$ denotes the $n$-th derivative of the Dedekind zeta function at $s=0$.\\

\item

Let $a$ and $m$ be relatively prime integers with $m \geq 1$. A famous result due to Dirichlet states that there exist infinitely many prime numbers $p$ such that
\[
p \equiv a\, (\mathrm{mod}\,{m}).
\]
See \cite{Serre_arithmetic}. This can be proved using the properties of the $L$-functions. \\ The goal of Section~\ref{ZPAP} is to give a formula for 
\[
\prod_{\substack{p\ \text{prime} \\ p \equiv a\, (\mathrm{mod}\,{m})}} p.
\]

For example, we obtain  
 \[
\begin{split}
\Rprod_{p \equiv 1\, (\mathrm{mod}\,{4})} p 
&= \exp\Bigg( 
    \log(\sqrt{\pi} \, G) \left( \frac{2}{\log 2} \log(2\pi) - 1 \right) \\
&\quad + \log\left( \frac{2\pi}{\sqrt{2}} \right) \left( \frac{3}{2} + \frac{1}{\log 2} \log(2\pi) \right)
\Bigg),
\end{split}
\]
  where $G$ is Gauss's constant. This example shows that there are infinitly many prime $p$ such that
  $p \equiv 1\, (\mathrm{mod}\,{4}).$\\

  \item In Section~\ref{schemefinite}, we consider zeta functions associated with closed points in smooth projective schemes over finite fields. Namely, let $X$ be a smooth projective scheme over a finite field $\mathbb{F}_q$ with $q$ elements. We introduce the Dirichlet series
\[
\mathcal{P}_X(s) = \sum_{x \in X^0} q^{-s \deg(x)},
\]
where the sum runs over the closed points of $X$, and $\deg(x)$ denotes the degree of the closed point $x$. Let $\widetilde{F}(t)$ be the function defined by
\[
\widetilde{F}(q^{-s}) = (1 - q^{-s}) \zeta(X,s),
\]

Let $C_X(0)$ denote the coefficient given by
\[
C_X(0) = \lim_{s \to 0} (1 - q^{-s}) \zeta(X,s).
\]
 It is known that $C_X(0) = \pm \chi(X,\mathbb{Z})$, with $\chi(X,\mathbb{Z})$ defined as the alternating product of the orders of certain cohomology groups of $X$. For details, see \cite{Geisser-2004, DixExposes, Milne1, Tate-zeta1, Tate-abelian, Tate-zeta2, Schneider-zeta}. 

\

Our main result of this section is  the following equation:
\[
\Rprod_{x \in X^0} N(x) = \exp\left(  1\pm 2\frac{\widetilde{F}'(1)}{\chi(X,\Z)} \right).
\]
Note that $\frac{\widetilde{F}'(1)}{\chi(X,\Z)}$ is a rational number.\\

Let $\mathcal F_3$ be the Fermat cubic curve giver by the equation $x^3+y^3+z^3=0$ over $\mathbb{F}_2$. 
An easy computation shows that
\[
\Rprod_{x \in \mathcal{F}_3^0} N(x) =\exp\left(\frac{7}{3}\right).
\]

Let $C$ be the hyperelliptic curve given by $y^2 + y = x^5 + 1$ over $\mathbb{F}_2$. We obtain
\[
\Rprod_{x \in C^0} N(x) =\exp\left(\frac{33}{5}\right).
\]
Note that $\exp\left(\frac{7}{3}\right)$ and $\exp\left(\frac{33}{5}\right)$ are irrational numbers; therefore, the Fermat cubic curve and the hyperelliptic curve defined above have infinitely many closed points.\\

In Paragraph~\ref{RFF}, we  specialize our formula to the particular case of global function fields in one variable over finite fields. As an example of computation, we consider  $K = \mathbb{F}_q(T)$, the field of rational fractions with coefficients in $\mathbb F_q$, and we obtain
that the regularized product over all nonzero monic irreducible polynomials $p$ of $\mathbb F_q[T]$ satisfies:
\[
\Rprod_{p} N(p) = \exp\left( \frac{3q-1}{ q-1} \right).
\]
Note that $\exp\left( \frac{3q-1}{q-1} \right)$ is an irrational number. We infer, analogously to the main application in \cite{Marco_prime}, that the number of irreducible polynomials in $\mathbb{F}_q[T]$ is infinite.\\

In Paragraph~\ref{RFFP}, we discuss the regularization of primes in an arithmetic progression over $\mathbb{F}[T]$. 
Building upon a foundational result of Weil~\cite{Weil_1945}, we recall key properties of $L$-functions associated with Dirichlet characters on $\mathbb{F}[T]$. 
Based on this, we explain how the function field analogue of Theorem~\ref{pam} admits a simplified form, yielding a more explicit formula.

\end{itemize}

\medskip

\section{Super-Regularization of infinite products}\label{SR}

This section is based on \cite{Marco_prime}; most of the proofs are essentially contained in, or slight generalizations of, those in \cite[\S~2]{Marco_prime}. We shall outline their approach and introduce our main new assumption, adopted in this paper, which enables the definition of the regularized product of degrees of closed points in an  {arithmetic}  scheme.\\

Let  $\lambda = (\lambda_n)_{n \geq 1}$ be a sequence  of complex numbers.  Its associated zeta function is defined as follows:
\[
\zeta_\lambda(s) = \sum_{n=1}^{\infty} \frac{1}{\lambda_n^{s}}.
\]

We assume that $\zeta_\lambda$ is absolutely convergent in the half plane $\operatorname{Re} s > s_0$. 
Mu\~noz and P\`erez considered an extension of $\zeta_\lambda$ to two complex variables into a double Dirichlet series:
\[
\zeta_\lambda(s,t) = \sum_{n,m=1}^{\infty}  \frac{c_{n,m}}{\lambda_n^{s} \mu_m^{t}},
\]
where $c_{n,m}$ are complex numbers.
They assumed that this series is absolutely convergent in $U_0 = \{ \operatorname{Re} s > s_0 \} \times \{ \operatorname{Re} t > t_0 \}$, and defines in this domain $U_0 \subset \mathbb{C}^2$ a meromorphic function, and that there exists $t_0 < 0$ and that
\[
\zeta_\lambda(s,0) = \zeta_\lambda(s).
\]

Now, $(s,t) \mapsto \frac{\partial \zeta_\lambda}{\partial s}(s,t)$ is meromorphic in $U_0$. 
They assumed that there exists $t_1 \geq t_0$ such that for all $t \in \mathbb{C}$ with $\operatorname{Re}(t) > t_1$, 
the meromorphic function of one complex variable 
\[
s \mapsto \frac{\partial \zeta_\lambda}{\partial s}(s,t),
\]
which is meromorphic in $\operatorname{Re}(s) > s_0$, extends meromorphically to a half-plane $\{ \operatorname{Re}(s) > s_1 \}$, where $s_1 < 0$ (i.e., a neighborhood of $s = 0$). 
They denoted this extension by
\[
s \mapsto \operatorname{ext}_s \frac{\partial \zeta_\lambda}{\partial s}(s,t).
\]

In this paper, we assume that
\[
\operatorname{Res}\left( \frac{1}{s} \operatorname{ext}_s \frac{\partial \zeta_\lambda}{\partial s}(s,t),\, s=0 \right),
\]
defined for $\operatorname{Re}(t) > t_1$, admits a meromorphic extension to a neighborhood of $\{ t = 0 \}$ that is not identically $\infty$. We denote this extension by
\[
t \mapsto \operatorname{ext}_t \left( \operatorname{Res}\left( \frac{1}{s} \operatorname{ext}_s \frac{\partial \zeta_\lambda}{\partial s}(s,t),\, s=0 \right) \right).
\]
Moreover, we suppose that, in some open neighborhood of $0$, that
\begin{equation}\label{assumption1}
\operatorname{ext}_t \left( \operatorname{Res}\left( \frac{1}{s} \operatorname{ext}_s \frac{\partial \zeta_\lambda}{\partial s}(s,t),\, s=0 \right) \right)
= \sum_{m \in \mathbb{Z}} c_m(\lambda) t^m,
\end{equation}
is a Laurent expansion.

\begin{definition}
The {super-regularized product} of the sequence $\lambda = (\lambda_n)_{n \geq 1}$ is by definition, provided that the limits and meromorphic extensions exist,
\[
\Rprod_{n} \lambda_n:=\exp\left(
\mathrm{Res}\Bigl({-} \frac{1}{t}\mathrm{ext}_t\Bigl(\mathrm{Res}\Bigl(\frac{1}{s} \mathrm{ext}_s\frac{\pt \zeta_\lambda}
{\pt s}(s,t) ,s=0   \Bigr) \Bigr),t=0\Bigr)\right).
\]
We write in short
\[
\Rprod_{n} \lambda_n:=\exp\left(
\mathrm{Res}\Bigl({-} \frac{1}{t}\Bigl(\mathrm{Res}\Bigl(\frac{1}{s} \frac{\pt \zeta_\lambda}
{\pt s}(s,t) ,s=0   \Bigr) \Bigr),t=0\Bigr)\right).
\]

\end{definition}

\begin{remark}
It is evident that this definition generalizes \cite[Definition~1]{Marco_prime}. It also extends the notion of regularized product introduced by Kurokawa and Wakayama in \cite{Kurokawa-2005} and \cite{Kurokawa-2004}.
\end{remark}

\cite[Proposition 2]{Marco_prime} extends easily to our context:
\begin{proposition}
Let $(s, t) \mapsto \hat{\zeta}_\lambda(s, t)$ be an alternative two-variable holomorphic extension of $s \mapsto \zeta_\lambda(s)$, sharing the same analytic properties as $\zeta_\lambda(s,t)$, such that the super-regularized product construction remains valid. Then
\[
\mathrm{Res}\Bigl({-} \frac{1}{t}\Bigl(\mathrm{Res}\Bigl(\frac{1}{s} \frac{\pt \zeta_\lambda}
{\pt s}(s,t) ,s=0   \Bigr) \Bigr),t=0\Bigr)=
\mathrm{Res}\Bigl({-} \frac{1}{t}\Bigl(\mathrm{Res}\Bigl(\frac{1}{s} \frac{\pt \hat{\zeta}_\lambda}
{\pt s}(s,t) ,s=0   \Bigr) \Bigr),t=0\Bigr),
\]
thus the super-regularization is independent of the choice of the complex extension.

In particular, the super-regularization coincides with the classical zeta-regularization when this last one is well defined.
\end{proposition}

\section{Zeta function of primes in a number field}\label{ZPNF}

Let $K$ be a number field, i.e., a finite extension of $\mathbb{Q}$. By Dedekind's theorem, the Dedekind zeta function $\zeta_K(s)$ is analytic around $s=0$ and  satisfies
\[
\zeta_K(s) = -\frac{h_K R_K}{w_K} s^{r_1 + r_2 - 1} + O(s^{r_1 + r_2}) \quad \text{as } s \to 0,
\]
where $h_K$ is the class number of $K$, $R_K$ is the regulator of $K$, $w_K$ is the number of roots of unity in $K$,
and $r_1$ and $r_2$ are the numbers of real and complex embeddings of $K$, respectively.  

Let us denote by $m(K)$ the coefficient of $s^{r_1 + r_2}$ in the Taylor expansion of $\zeta_K(s)$ at $s = 0$. Namely,
\[
\zeta_K(s) = -\frac{h_K R_K}{w_K} s^{r_1 + r_2 - 1} + m(K) s^{r_1 + r_2} + O(s^{r_1 + r_2 + 1}) \quad \text{as } s \to 0.
\]
Namely, $m(K)=\frac{1}{(r_1+r_2)!}\zeta_K^{(r_1+r_2)}(0)$.\\

While the leading term (involving the class number, regulator, and number of roots of unity) is well understood, much less is known about the coefficient $m(K)$. One objective of this paper is to investigate the  {arithmetic}  significance of this coefficient.

\begin{definition} Let $K$ be a number field.
We let
\[
\mathcal{P}_K(s):=\sum_{\frak p } \frac{1}{(N \frak p)^s}
\]
and
\[
\mathcal{P}_K(s,t):=
\sum_{n=1}^\infty \frac{\mu(n)}{n^{t+1}}\log \zeta_K(ns)
\]
for $\mathrm{Re}(s)>1$ and $ \mathrm{Re}(t)>0$.
\end{definition}

\begin{proposition}\label{AbsConv}
\begin{equation}\label{defPK}
\mathcal{P}_K(s,t):=
\sum_{n=1}^\infty \frac{\mu(n)}{n^{t+1}}\log \zeta_K(ns)
\end{equation}
is absolutely and uniformly convergent for $\operatorname{Re}(s) \geq 1 + \delta$ and 
$\mathrm{Re}(t)\geq \eta$ where $\delta$ is a positive real number and 
$\eta$ is real number. 

\end{proposition} 

\begin{proof} Let us recall the proof of absolute convergence of $L$-functions associated with a Dirichlet character.  Let $\chi$ be a Dirichlet character. We set \[ E(s) = \prod_{\mathfrak{p}} \frac{1}{1 - \chi(\mathfrak{p}) N(\mathfrak{p})^{-s}}. \]
Taking formally the logarithm of the product, it gives the series
\[ \log E(ns) = \sum_{\mathfrak{p}} \sum_{k=1}^{\infty} \frac{\chi(\mathfrak{p})^k}{k N(\mathfrak{p})^{kns}}. \]
Now let us consider the following series
\[
\sum_{n=1}^\infty \frac{\mu(n)}{n^{t+1}}\log E(ns)=\sum_{n=1}^\infty \frac{\mu(n)}{n^{t+1}}
\sum_{\mathfrak{p}} \sum_{k=1}^{\infty} \frac{\chi(\mathfrak{p})^k}{k N(\mathfrak{p})^{kns}}.
\]

It can be shown that this series converges absolutely and uniformly for $\operatorname{Re}(ns) = n\sigma \geq n(1 + \delta)$, where $\delta > 0$ is fixed and independent of $n$. The proof is a slight generalization of the convergence argument for $L$-functions given in \cite{Neukirch}.\\

The absolute convergence of \eqref{defPK} follows from the  following equation: 
\[ \sum_{n=1}^\infty \frac{|\mu(n)|}{n^{\mathrm{Re}(t)+1}}
\sum_{\mathfrak{p}, k} \frac{[K:\Q]}{kp^{kn(1+\mathrm{Re}(s))}} = [K:\Q] 
 \sum_{n=1}^\infty \frac{|\mu(n)|}{n^{\mathrm{Re}(t)+1}}
\log \zeta(n(1+\mathrm{Re}(s))). \]
We conclude using that
\[
\log\zeta(n(\delta+1))=\frac{1}{2^{n(\delta+1)}}O(1).
\]

\end{proof}
From the identity
\[
\exp(X) = \prod_{n=1}^{+\infty} (1 - X^n)^{-\frac{\mu(n)}{n}},
\]
where $\mu$ is the M\"obius function, we obtain for every prime ideal $\mathfrak{p}$ of $K$:
\[
e^{N(\mathfrak{p})^{-s}} = \prod_{n=1}^{+\infty} (1 - N(\mathfrak{p})^{-ns})^{-\frac{\mu(n)}{n}}.
\]
It follows that
\[
\begin{split}
e^{\mathcal{P}_K(s)} &= \prod_{\mathfrak{p}} e^{N(\mathfrak{p})^{-s}} \\
&= \prod_{\mathfrak{p}} \prod_{n=1}^{+\infty} (1 - N(\mathfrak{p})^{-ns})^{-\frac{\mu(n)}{n}} \\
&= \prod_{n=1}^{+\infty} \prod_{\mathfrak{p}} (1 - N(\mathfrak{p})^{-ns})^{-\frac{\mu(n)}{n}} \\
&= \prod_{n=1}^{+\infty} \zeta_K(ns)^{\frac{\mu(n)}{n}}.
\end{split}
\]
The interchange of products is justified by Proposition~\ref{AbsConv}. Hence,
\[
\mathcal{P}_K(s) = \sum_{n=1}^\infty \frac{\mu(n)}{n} \log \zeta_K(ns), \quad \operatorname{Re}(s) > 1.
\]

We note that
\[
\mathcal{P}_K(s,0)=\mathcal{P}_K(s)\quad \mathrm{Re}(s)>1.
\]
Also, observe that the following holds
\[
\frac{\pt}{\pt s} \mathcal{P}_K(s,t)=\sum_{n=1}^\infty \frac{\mu(n)}{n^{t}}\frac{\zeta_K'(ns)}{\zeta_K(ns)}\quad \quad \forall \mathrm{Re}(s)>1, \ \forall  t\in \C.
\]

Since
\[
\mathrm{Res}\left( {\frac{1}{s}} \frac{\zeta_K'(s)}{\zeta_K(s)},s=0\right)= \frac{1}{r_1+r_2} \frac{\zeta_K^{(r_1+r_2)}(0) }{\zeta_K^{(r_1+r_2-1)}(0)}.
\]

Then
\[
\mathrm{Res}\left(\frac{1}{t}\mathrm{Res}\left( \frac{1}{s}\mathrm{ext}_s\frac{\pt \mathcal P_K}{\pt s} (s,t),s=0\right), t=0\right)=\frac{2 w_K}{(r_1+r_2)h_K R_K} \zeta_K^{(r_1+r_2)}(0) .
\]
Hence, we have proved the following result.
\begin{theorem}\label{NF}
The regularized product over all nonzero prime ideals $\mathfrak{p} \in \mathrm{Spec}(\mathcal{O}_K) \setminus \{0\}$ satisfies:
\[
\Rprod_{\mathfrak{p} \in \mathrm{Spec}(\mathcal{O}_K) \setminus \{0\}} N\mathfrak{p} = \exp\left( -\frac{2 w_K}{(r_1 + r_2) h_K R_K}  \zeta_K^{(r_1 + r_2)}(0) \right),
\]
where
$w_K$ is the number of roots of unity in $K$, $r_1$, $r_2$ are the numbers of real and complex embeddings of $K$,
  $h_K$ is the class number,
  $R_K$ is the regulator, and 
  $\zeta_K^{(n)}(0)$ denotes the $n$-th derivative of the Dedekind zeta function at $s=0$.
\end{theorem}


\section{Zeta functions associated with  {arithmetic}  progressions}\label{ZPAP}

The main goal of this section is to determine the following regularized product:
\[
\Rprod_{\substack{p\ \text{prime} \\ p \equiv a\, \mathrm{mod}\, m}} p,
\]
 where $a$ and $m$ are positive integer with $\gcd(a,m)=1$.\\

\begin{definition}
Let $a$ be a nonzero integer with $\gcd(a, m) = 1$. We define
\[
\mathcal{P}_{m,a}(s) := \sum_{\substack{p \text{ prime} \\ p \equiv a\, \mathrm{mod}\, m}} \frac{1}{p^s},
\]
and for any Dirichlet character $\chi$ modulo $m$,
\[
\mathcal{P}_\chi(s) := \sum_{p} \frac{\chi(p)}{p^s}.
\]

\end{definition}

Throughout this paper, $\mu$ denotes the M\"obius function, $\omega(n)$ the number of distinct prime factors of a positive integer $n$, and $\zeta_\mathbb{Q}$ the Riemann zeta function.  The Hurwitz zeta function $\zeta_H(s,x)$ is defined by
\[
\zeta_H(s,x) = \sum_{k=0}^\infty \frac{1}{(k + x)^s},
\]
for $\operatorname{Re}(s) > 1$ and $x > 0$. We denote the Riemann zeta function by $\zeta_\mathbb{Q}(s)$, and recall that $\zeta_\mathbb{Q}(s) = \zeta_H(s,1)$.\\

 We recall the following classical identities:

\[
\frac{1}{\zeta_\mathbb{Q}(s)} = \sum_{n=1}^\infty \frac{\mu(n)}{n^s},
\quad \text{\cite[p.~229]{ApostolNumberTheory}}.
\]

\[
-\frac{\zeta_\mathbb{Q}'(s)}{\zeta_\mathbb{Q}(s)} = \sum_{n=2}^\infty \frac{\Lambda(n)}{n^s},
\quad \text{\cite[p.~236]{ApostolNumberTheory}},
\]
where $\Lambda(n)$ denotes the von Mangoldt function \cite[p.~32]{ApostolNumberTheory}, defined by
\[
\Lambda(n) =
\begin{cases}
\log p & \text{if } n = p^a \text{ for some prime } p \text{ and integer } a \geq 1, \\
0 & \text{otherwise}.
\end{cases}
\]

For a Dirichlet character $\chi$, the associated $L$-function is given by the Dirichlet series:
\[
L(s,\chi) = \sum_{n=1}^\infty \frac{\chi(n)}{n^s},
\]
and admits an Euler product expansion over primes:
\[
L(s,\chi) = \prod_{p} \frac{1}{1 - \frac{\chi(p)}{p^s}}.
\]

In the case of the principal character $\chi_0 \bmod m$, we have:
\[
L(s,\chi_0) = \zeta_\mathbb{Q}(s) \prod_{p \mid m} \left(1 - \frac{1}{p^s}\right).
\]

Finally, the logarithmic derivative satisfies:
\[
\frac{L'(s,\chi)}{L(s,\chi)} = -\sum_{n=2}^\infty \frac{\Lambda(n) \chi(n)}{n^s}.
\]

It is well known that for any Dirichlet character $\chi$, the reciprocal of the $L$-function admits a Dirichlet series representation:
\[
\frac{1}{L(s, \chi)} = \sum_{n=1}^\infty \frac{\mu(n) \chi(n)}{n^s}.
\]
This identity follows from the Euler product expansion of $L(s,\chi)$ and the multiplicative inversion of Dirichlet series (see  \cite[p.~229]{ApostolNumberTheory}).\\

We recall the generating function:
\[
\frac{s e^{sx}}{e^s - 1} = \sum_{k=0}^\infty B_k(x) \frac{s^k}{k!}.
\]
The coefficient $B_k(x)$ (resp.~$B_k = B_k(0)$) is the $k$-th Bernoulli polynomial (resp.~Bernoulli number). For example,
\[
B_0(x) = 1, \quad B_1(x) = x - \frac{1}{2}, \quad B_2(x) = x^2 - x + \frac{1}{6},
\]
and
\[
B_0 = 1, \quad B_1 = -\frac{1}{2}, \quad B_2 = \frac{1}{6},
\]
see for example \cite[p.~27]{Erdelyi1}.
For positive integers $k$, we have the classical relation:
\[
\zeta_H(1 - k, x) = -\frac{B_k(x)}{k}, \quad k = 1,2,\ldots
\]

Additionally, Lerch's identity gives:
\[
\zeta_H'(0, x) = -\log \left( \frac{\sqrt{2\pi}}{\Gamma(x)} \right), \quad x > 0.
\]

In particular, for the Riemann zeta function:
\[
\zeta_\mathbb{Q}(0) = -\frac{1}{2}, \quad \zeta_\mathbb{Q}'(0) = -\frac{1}{2} \log(2\pi).
\]

We define the generalized Bernoulli polynomials $B_k^{[\ell]}(x; \alpha_1,\ldots,\alpha_\ell; a_1,\ldots,a_\ell)$ via the generating function:
\[
\frac{s^{\ell} e^{xs}}{\prod_{j=1}^\ell (\alpha_j e^{a_j s} - 1)} = \sum_{k=0}^\infty B_k^{[\ell]}(x; \alpha_1,\ldots,\alpha_\ell; a_1,\ldots,a_\ell) \frac{s^k}{k!}, \quad \text{for } |s| \ll 1,
\]
where $\alpha_1,\ldots,\alpha_\ell$, $a_1,\ldots,a_\ell$, and $x$ are complex numbers. We further define the corresponding generalized Bernoulli numbers by setting
\[
B_k^{[\ell]}(\alpha_1,\ldots,\alpha_\ell; a_1,\ldots,a_\ell) := B_k^{[\ell]}(0; \alpha_1,\ldots,\alpha_\ell; a_1,\ldots,a_\ell).
\]

Let $\chi$ be Dirichlet character of conductor $f$. Let
\[
\sum_{a=1}^f \chi(a) \frac{t e^{at}}{e^{ft} - 1} = \sum_{n=0}^\infty B_{n,\chi} \frac{t^n}{n!}.
\]
The coefficients $B_{n,\chi}$ are called the \emph{generalized Bernoulli numbers}.

 \begin{proposition}
 Let $\chi$ be a Dirichlet character. Define
\[
\delta_\chi =
\begin{cases}
0 & \text{if } \chi(-1) = 1, \\
1 & \text{if } \chi(-1) = -1.
\end{cases}
\]

\begin{enumerate}
    \item If $\chi = \chi_0$ (the principal character), then:
    \[
    \begin{split}
        B_0 &= 1, \quad B_n \ne 0 \text{ for even } n \geq 0, \\
        B_1 &= \frac{1}{2}, \quad B_n = 0 \text{ for odd } n > 1.
    \end{split}
    \]

    \item If $\chi \ne \chi_0$, then the generalized Bernoulli numbers $B_{n,\chi}$ satisfy:
    \[
    \begin{split}
        B_{0,\chi} &= 0, \\
        B_{n,\chi} &\ne 0 \quad \text{for } n \geq 1 \text{ and } n \equiv \delta_\chi \pmod{2}, \\
        B_{n,\chi} &= 0 \quad \text{for } n \geq 1 \text{ and } n \not\equiv \delta_\chi \pmod{2}.
    \end{split}
    \]
    Moreover, for all integers $n \geq 1$,
    \[
    L(1 - n, \chi) = -\frac{B_{n,\chi}}{n}.
    \]
\end{enumerate}

 \end{proposition}
 \begin{proof}
 See \cite[Theorem 2, p. 12]{Iwasawa} or \cite[p. 268]{ApostolNumberTheory}.\\
 \end{proof}

\begin{lemma}\label{formulas1}
The Dirichlet $L$-function $L(s, \chi_0)$ admits a meromorphic continuation to the whole complex plane, which is holomorphic at $s=0$. Moreover,

\[
\mathrm{ord}_{s=0} L(s, \chi_0) = \omega(m).
\]

The derivative at $s=0$ satisfies:
\[
L'(0, \chi_0) =
\begin{cases}
-\dfrac{1}{2} \log p_1 & \text{if } \omega(m) = 1, \\
0 & \text{if } \omega(m) \geq 2.
\end{cases}
\]

For $0 < |t| \ll 1$, we have the Laurent expansion:
\[
\frac{1}{t L(t, \eta)} = \sum_{n \geq -\omega(m) - 1} \mathfrak{t}_n(m; \eta) t^n.
\]

Moreover,
\begin{equation}\label{eq:laurent-L}
\frac{1}{t L(t, \eta)} = \sum_{n} \mathfrak{t}_n(m; \eta) t^n =
\begin{cases}
\dfrac{1}{L(0, \eta)\, t} + O(1) & \text{if } \eta(-1) = -1, \\[2ex]
\dfrac{1}{L'(0, \eta)\, t^2} - \dfrac{L''(0, \eta)}{2 L'(0, \eta)^2\, t} + O(1) & \text{if } \eta(-1) = 1 \text{ and } \eta \neq \chi_0.
\end{cases}
\end{equation}

When $\eta = \chi_0$, we have:
\begin{equation}\label{bl}
\begin{split}
\mathrm{Res}\left( \frac{1}{s L(s, \chi_0)}, s=0 \right)
&= \frac{1}{\omega(m)!} \sum_{j=0}^{\omega(m)} \binom{\omega(m)}{j} 
\left. \frac{\partial^j}{\partial s^j} \left( \frac{1}{\zeta_{\mathbb{Q}}(s)} \right) \right|_{s=0}
B_{\omega(m)-j}(-\log p_1, \ldots, -\log p_{\omega(m)}) \\
&=: b_{\omega(m)},
\end{split}
\end{equation}
where $B_k(x_1,\dots,x_r)$ denotes the complete Bell polynomial in $r$ variables.

Additionally, the Laurent expansion around $s=0$ is:
\[
\frac{1}{s L(s, \chi_0)} = \sum_{n = -\omega(m) - 1}^{\infty} \mathfrak{t}_n(m; \chi_0) s^n.
\]
\end{lemma}

 \begin{proof}

This is a classical computation and is therefore left as an exercise.
  \end{proof}

We recall the following classical identities, which can be found \cite{ApostolNumberTheory}:
\begin{proposition}
Let $\chi$ be a Dirichlet character modulo $f$. Then:
\begin{enumerate}
    \item If $\chi = \chi_0$ (the principal character), then
    \[
    L(0, \chi_0) = 0.
    \]

    \item If $\chi \neq \chi_0$, then
    \[
    L(0, \chi) = -\frac{1}{f} \sum_{a=1}^f \chi(a) a.
    \]
    Moreover,
    \begin{itemize}
        \item $L(0, \chi) = 0$ if $\chi(-1) = 1$,
        \item $L(0, \chi) \ne 0$ if $\chi(-1) = -1$.
    \end{itemize}
\end{enumerate}
\end{proposition}

\bigskip

Let $m$ be a positive integer $\geq 2$.  Let $p_1,\ldots,p_\ell$, where $\ell = \omega(m)$, be the distinct prime divisors of $m$. Let $\chi$ be a Dirichlet character modulo $m$, with conductor $f = f_\chi$. We let
\[
G(m) := (\mathbb{Z}/m\mathbb{Z})^\ast \quad \text{and} \quad \varphi(m) := \#(\mathbb{Z}/m\mathbb{Z})^\ast.
\]
Let
\[
\widehat{G(m)}
\]
denote the group of Dirichlet characters modulo $m$. We denote by $\chi_0$ the principal Dirichlet character modulo $m$. \\

Let $n$ be a nonnegative integer. We consider three subsets of $\widehat{G(m)}$:

\begin{itemize}
    \item 
    \[
    \widehat{G(m)}_{n,+} := \left\{ \chi \in \widehat{G(m)} \ \middle|\ \chi^n(-1) = 1 \right\},
    \]
    \item 
    \[
    \widehat{G(m)}_{n,-} := \left\{ \chi \in \widehat{G(m)} \ \middle|\ \chi^n(-1) = -1 \right\}.
    \]
    Obviously, $\widehat{G(m)}_{n,-}$ is empty if $n$ is even, and equals the set of characters $\chi$ such that $\chi(-1) = -1$ when $n$ is odd.
    \item 
    \[
    \widehat{G(m)}_{n,0} := \left\{ \chi \in \widehat{G(m)} \ \middle|\ \chi^n = \chi_0 \right\}.
    \]
\end{itemize}

\begin{definition}
Let $m \geq 2$. For $a \in \mathbb{Z}$ with $\gcd(a,m) = 1$, and for every $n \in \mathbb{N}_{\geq 1}$, define:
\[
\mathscr{Q}_n^0(m;a) := \frac{1}{\varphi(m)} \sum_{\chi \in \widehat{G(m)}_{n,0}} \chi^{-1}(a)  \mathrm{Res}\left( \frac{L'(s,\chi^n)}{s L(s,\chi^n)}, s=0 \right),
\]
\[
\mathscr{Q}_n^-(m;a) := \frac{1}{\varphi(m)} \sum_{\chi \in \widehat{G(m)}_{n,-}} \chi^{-1}(a) \mathrm{Res}\left( \frac{L'(s,\chi^n)}{s L(s,\chi^n)}, s=0 \right),
\]
and
\[
\mathscr{Q}_n^+(m;a) := 
\begin{cases}
\displaystyle
\frac{1}{\varphi(m)} \sum_{\chi \in \widehat{G(m)}_{n,+} \setminus \{\chi_0\}} \chi^{-1}(a) \cdot \mathrm{Res}\left( \frac{L'(s,\chi^n)}{s L(s,\chi^n)}, s=0 \right), & \text{if } \widehat{G(m)}_{n,+} \neq \{\chi_0\}, \\
0, & \text{otherwise}.
\end{cases}
\]
\end{definition}


\begin{proposition}

\item 

\begin{enumerate}[label=\roman*)]

\item 
\[
\mathscr{Q}_n^0(m;a)=
\frac{1}{\vf(m)}\Bigl(\sum_{\substack{\chi\in {\widehat{G(m)}} \\
\chi^n=\chi_0}}  \chi^{-1}(a) \Bigr) \log \frac{2\pi}{\left(\prod_{j=1}^{\ell}p_j \right)^{\frac{1}{2}}}.
\]
where $p_1,\ldots,p_\ell$ are the distinct prime divisors of $m$.
\item 
\[
\mathscr{Q}_n^-(m;a)=\frac{1}{\vf(m)}\sum_{\chi\in \widehat{G(m)}_{n,-}} \chi^{-1}(a)   \frac{L'(0,\chi^n)}{L(0,\chi^n)}.
\]

\item 
\[
\mathscr{Q}_n^+(m;a)=
\frac{1}{\varphi(m)} \sum_{\chi \in \widehat{G(m)}_{n,+} \setminus \{\chi_0\}} \chi^{-1}(a) \frac{L^{(r_\chi)}(s,\chi^n)}{ (r_\chi+1)L^{(r_\chi+1)}(0,\chi^n)}, 
\]
where $r_\chi$ denotes the order of vanishing of $L(s, \chi^n)$ at $s = 0$.

\end{enumerate}
\end{proposition}

\begin{proof}
These are basic identities that follow from
\[
\operatorname{Res}\left( \frac{L'(s,\chi)}{s L(s,\chi)}, s=0 \right) = \frac{L^{(r+1)}(0,\chi)}{(r+1) L^{(r)}(0,\chi)},
\]
where $r$ is the order of vanishing of $L(s,\chi)$ at $s=0$.

In the case of the principal character $\chi_0$, by considering the meromorphic continuation of $L(s,\chi_0)$, we obtain, in an open neighborhood of $s=0$,
\begin{equation}\label{ls0}
L(s,\chi_0) = \zeta_\mathbb{Q}(0) \left( \prod_{i=1}^{\ell} \log p_i \right) s^\ell + \left( \prod_{i=1}^{\ell} \log p_i \right) \left( \zeta_\mathbb{Q}'(0) - \frac{1}{2} \zeta_\mathbb{Q}(0) \log \prod_{j=1}^{\ell} p_j \right) s^{\ell+1} + O(s^{\ell+2}).
\end{equation}
Here, $\ell$ is the order of vanishing of $L(s,\chi_0)$ at $s=0$, which equals the number of distinct prime divisors of $m$, i.e., $\ell = \omega(m)$.

In particular,
\begin{equation}\label{8}
\frac{L^{(\ell+1)}(0,\chi_0)}{(\ell+1) L^{(\ell)}(0,\chi_0)} = \frac{\zeta_\mathbb{Q}'(0)}{\zeta_\mathbb{Q}(0)} - \frac{1}{2} \log \prod_{j=1}^{\ell} p_j.
\end{equation}
\end{proof}

\begin{definition}\label{def:RNR} Let $N$ be a positive integer.
For every $r = 1, 2, \ldots, N$, we define:
\[
\mathscr{R}_{N,r} := 
\begin{cases}
0, & \text{if } \mu(\gcd(r,N)) = 0, \\
\displaystyle
\sum_{\eta \in \widehat{G\left( \frac{N}{\gcd(N,r)} \right)}} 
\overline{\eta}\left( \frac{r}{\gcd(r,N)} \right) 
\sum_{-\frak{o}_\eta \leq j \leq \omega(\gcd(r,N))} 
\frak{b}_{-j}(\gcd(r,N); \eta; \mathbf{p}) \cdot \frak{l}_j(\eta), & \text{otherwise}.
\end{cases}
\]

Here:

\begin{itemize}
    \item $\mathbf{p} = (p_1, \dots, p_{\omega(d)})$ denotes the tuple of distinct prime divisors of $d = \gcd(r,N)$.
    
    \item For $k \geq -\omega(d)$, we define
    \[
    \frak{b}_k(d; \eta; \mathbf{p}) := 
    \frac{1}{(k + \omega(d))!} 
    B_{k+\omega(d)}\bigl(-\log d; \eta(p_1), \dots, \eta(p_{\omega(d)}); -\log p_1, \dots, -\log p_{\omega(d)}\bigr),
    \]
    where $B_n(x; \alpha_1,\dots,\alpha_m; \beta_1,\dots,\beta_m)$ denotes the generalized Bernoulli polynomial associated to the parameters $(\alpha_i, \beta_i)$.

    \item $\frak{o}_\eta$ is the order of vanishing of $L(t, \eta)$ at $t=0$, i.e., the smallest integer such that the Laurent expansion of $1/L(t,\eta)$ at $t=0$ begins with $t^{-\frak{o}_\eta}$.

    \item The coefficients $\frak{l}_j(\eta)$ are defined by the expansion:
    \[
    \frac{1}{L(t, \eta)} = \sum_{j \geq -\frak{o}_\eta} \frak{l}_j(\eta) \, t^j,
    \]
    valid in a neighborhood of $t=0$.
\end{itemize}
\end{definition}

\begin{proposition}\label{rest}
Let $N \geq 2$ be a positive integer. For each $r = 1, 2, \dots, N$, define
\[
\mathscr{R}_{N,r} := 
\operatorname{Res}_{t=0} \left( \frac{1}{t} \sum_{\ell=0}^\infty \frac{\mu(N\ell + r)}{(N\ell + r)^t} \right).
\]
Then
\[
\mathscr{R}_{N,r} = 
\sum_{\eta \in \widehat{G\left( \frac{N}{d} \right)}} 
\overline{\eta}\left( \frac{r}{d} \right) 
\sum_{-\frak{o}_\eta \leq j \leq \omega(d)} 
\frak{b}_{-j}(d; \eta; \mathbf{p}) \cdot \frak{l}_j(\eta),
\]
where $d = \gcd(N,r)$, $\mathbf{p}$ denotes the tuple of distinct prime divisors of $d$, $\frak{o}_\eta$ is the order of vanishing of $L(t,\eta)$ at $t=0$, and $\frak{b}_k$, $\frak{l}_j$ are as defined in Definition~\ref{def:RNR}.

In particular, when $\gcd(r,N) = 1$,
\begin{equation}\label{casecoprime}
\begin{split}
\mathscr{R}_{N,r} = 
&\frac{1}{\varphi(N)} \sum_{\substack{\eta \in \widehat{G(N)} \\ \eta(-1) = -1}} \overline{\eta}(r) \cdot L(0,\eta)^{-1} \\
&- \frac{1}{\varphi(N)} \sum_{\substack{\eta \in \widehat{G(N)} \\ \eta(-1) = +1 \\ \eta \ne \eta_0}} \overline{\eta}(r) \cdot \frac{L''(0,\eta)}{2 L'(0,\eta)^2} \\
&+ \frac{1}{\varphi(N)} \cdot b_{\omega(N)},
\end{split}
\end{equation}
where:
\begin{itemize}
    \item $\eta_0$ is the principal Dirichlet character modulo $N$,
    \item $\omega(N)$ is the number of distinct prime divisors of $N$,
    \item $b_{\omega(N)}$ is defined as in \eqref{bl},
    \item $\widehat{G(N)}_{1,\pm}$ denotes the set of characters modulo $N$ satisfying $\eta(-1) = \pm 1$, excluding $\eta_0$ in the $+$ case.
\end{itemize}
\end{proposition}

\begin{proof}

\begin{enumerate}[label= \roman*)]

\item 

Let $\eta$ be a Dirichlet character modulo $\frac{N}{d}$, and let $\delta_0$ denote the principal character modulo $d$. Define the Dirichlet series
\[
L(t, \eta \delta_0) := \sum_{n=1}^\infty \frac{\eta(n) \delta_0(n)}{n^t}.
\]
Since both $\eta$ and $\delta_0$ are multiplicative, their product $\eta \delta_0$ is completely multiplicative. \\

Hence, the Euler product representation holds:
\[
L(t, \eta \delta_0) = \prod_p \left( 1 - \frac{\eta(p) \delta_0(p)}{p^t} \right)^{-1}.
\]

Note that $\delta_0(p) = 0$ if $p \mid d$, and $\delta_0(p) = 1$ otherwise. Therefore,
\[
L(t, \eta \delta_0) = L(t, \eta) \cdot \prod_{p \mid d} \left( 1 - \frac{\eta(p)}{p^t} \right).
\]

On the other hand, by M\"obius inversion for multiplicative functions, we have:
\[
\frac{1}{L(t, \eta \delta_0)} = \sum_{n=1}^\infty \frac{\mu(n) \eta(n) \delta_0(n)}{n^t}.
\]

Now, using orthogonality of characters modulo $\frac{N}{d}$, we obtain:
\[
\sum_{\eta \in \widehat{G\left( \frac{N}{d} \right)}} \overline{\eta}\left( \frac{r}{d} \right) \cdot \frac{1}{L(t, \eta \delta_0)} 
= \sum_{\substack{n=1 \\ n \equiv \frac{r}{d}\pmod{\frac{N}{d}}}}^\infty \frac{\mu(n) \delta_0(n)}{n^t}.
\]

Substituting the expression for $L(t, \eta \delta_0)$ from above, we get:
\[
\sum_{\substack{n=1 \\ n \equiv \frac{r}{d}\pmod{\frac{N}{d}}}}^\infty \frac{\mu(n) \delta_0(n)}{n^t}
= \sum_{\eta \in \widehat{G\left( \frac{N}{d} \right)}} 
\frac{ \overline{\eta}\left( \frac{r}{d} \right) }{ L(t, \eta) \cdot \displaystyle\prod_{p \mid d} \left( 1 - \frac{\eta(p)}{p^t} \right) }.
\]

\begin{equation*}
\sum_{\ell=0}^\infty \frac{\mu(N\ell + r)}{(N\ell + r)^t}
= \frac{\mu(d)}{d^t} \sum_{\substack{l=0 \\ (d,\, \frac{N}{d}\ell + \frac{r}{d}) = 1}}^\infty \frac{\mu\left( \frac{N}{d}\ell + \frac{r}{d} \right)}{\left( \frac{N}{d}\ell + \frac{r}{d} \right)^t},
\end{equation*}
where $d = \gcd(N,r)$. The condition $(d, \frac{N}{d}\ell + \frac{r}{d}) = 1$ ensures $\mu(N\ell + r) \ne 0$, since otherwise $N\ell + r$ has a square factor.

Rewriting via congruence:
\begin{equation*}
= \frac{\mu(d)}{d^t} \sum_{\substack{n=1 \\ n \equiv \frac{r}{d}\pmod{\frac{N}{d}}}}^\infty \frac{\mu(n) \delta_0(n)}{n^t}.
\end{equation*}

By orthogonality:
\begin{equation*}
= \frac{\mu(d)}{d^t \varphi\left( \frac{N}{d} \right)} \sum_{\eta \in \widehat{G\left( \frac{N}{d} \right)}} \overline{\eta}\left( \frac{r}{d} \right) \cdot \frac{1}{L(t, \eta \cdot \delta_0)}.
\end{equation*}

Using $L(t, \eta \cdot \delta_0) = L(t, \eta) \prod_{p \mid d} \left( 1 - \frac{\eta(p)}{p^t} \right)$,
\begin{equation*}
= \frac{\mu(d)}{d^t \varphi\left( \frac{N}{d} \right)} \sum_{\eta \in \widehat{G\left( \frac{N}{d} \right)}} \frac{ \overline{\eta}\left( \frac{r}{d} \right) }{ L(t, \eta) \cdot \displaystyle\prod_{p \mid d} \left( 1 - \frac{\eta(p)}{p^t} \right) }.
\end{equation*}

Equivalently,
\begin{equation*}
= \frac{\mu(d)}{\varphi\left( \frac{N}{d} \right)} \sum_{\eta \in \widehat{G\left( \frac{N}{d} \right)}} \frac{ e^{-t \log d} }{ \displaystyle\prod_{p \mid d} \left( 1 - \eta(p) e^{-t \log p} \right) } \cdot \frac{ \overline{\eta}\left( \frac{r}{d} \right) }{ L(t, \eta) }.
\end{equation*}

Let $p_1, \dots, p_{\omega(d)}$ be the distinct prime divisors of $d$. Define
\begin{equation*}
F_\eta(t) := \frac{ e^{-t \log d} }{ \displaystyle\prod_{p \mid d} \left( 1 - \eta(p) e^{-t \log p} \right) } = \sum_{k \geq -\omega(d)} \frak{b}_k(d; \eta; \mathbf{p}) \, t^k,
\end{equation*}
where
\begin{equation*}
\frak{b}_k(d; \eta; \mathbf{p}) := \frac{1}{(k + \omega(d))!} 
B_{k+\omega(d)}^{[\omega(d)]}(-\log d; \eta(p_1), \dots, \eta(p_{\omega(d)}); -\log p_1, \dots, -\log p_{\omega(d)}),
\end{equation*}
and $B_n^{[m]}$ denotes the generalized Bernoulli polynomial of order $m$.\\

Let $\frac{1}{L(t, \eta)} = \sum_{j \geq -\frak{o}_\eta} \frak{l}_j(\eta) \, t^j$, where $\frak{o}_\eta$ is the order of vanishing of $L(t,\eta)$ at $t=0$.

Then the constant term of the full series is
\begin{equation*}
\frac{\mu(d)}{\varphi\left( \frac{N}{d} \right)} \sum_{\eta \in \widehat{G\left( \frac{N}{d} \right)}} \overline{\eta}\left( \frac{r}{d} \right) \sum_{-\frak{o}_\eta \leq j \leq \omega(d)} \frak{b}_{-j}(d; \eta; \mathbf{p}) \cdot \frak{l}_j(\eta),
\end{equation*}
which equals $\mathscr{R}_{N,r}$ by definition.\\

\item When $\gcd(N,r)=1$, we have
\[
\begin{split}
\sum_{\ell=0}^\infty \frac{\mu(N\ell + r)}{(N\ell + r)^t}
=& \sum_{\eta \in \widehat{G(N)}} \frac{\overline{\eta}(r)}{L(t,\eta)}\\
=& \sum_{\eta \in \widehat{G(N)}_{1,-}} \frac{\overline{\eta}(r)}{L(t,\eta)}
+ \sum_{\eta \in \widehat{G(N)}_{1,+}} \frac{\overline{\eta}(r)}{L(t,\eta)}
+ \sum_{\eta \in \widehat{G(N)}_{1,0}} \frac{\overline{\eta}(r)}{L(t,\eta)}.
\end{split}
\]

For $\mathrm{Re}(t) > 1$, using $L'(0,\eta) \ne 0$ for $\eta(-1) = 1$, $\eta \ne \eta_0$,
\begin{align*}
\frac{1}{t} \sum_{\ell=0}^\infty \frac{\mu(N\ell + r)}{(N\ell + r)^t}
=& \frac{1}{\varphi(N)} \sum_{\eta \in \widehat{G(N)}} \frac{\overline{\eta}(r)}{t L(t,\eta)} \\
=& \frac{1}{\varphi(N)} \sum_{\eta \in \widehat{G(N)}_{1,-}} \frac{\overline{\eta}(r)}{t L(t,\eta)}
+ \frac{1}{\varphi(N)} \sum_{\eta \in \widehat{G(N)}_{1,+}} \frac{\overline{\eta}(r)}{t L(t,\eta)}
+ \frac{1}{\varphi(N)} \cdot \frac{1}{t L(t,\eta_0)} \\
=& \frac{1}{\varphi(N)} \sum_{\eta \in \widehat{G(N)}_{1,-}} \frac{\overline{\eta}(r)}{L(0,\eta) t}
+ \frac{1}{\varphi(N)} \sum_{\eta \in \widehat{G(N)}_{1,+}} \frac{\overline{\eta}(r)}{L'(0,\eta)}  \frac{1}{t^2}
\\
&- \frac{1}{\varphi(N)} \sum_{\eta \in \widehat{G(N)}_{1,+}} \frac{\overline{\eta}(r) L''(0,\eta)}{2 L'(0,\eta)^2}  \frac{1}{t}  + \frac{1}{\varphi(N)}  \frac{1}{t L(t,\eta_0)}.
\end{align*}

Hence,
\[
\begin{split}
\operatorname{Res} \left( \frac{1}{t} \sum_{\ell=0}^\infty \frac{\mu(N\ell + r)}{(N\ell + r)^t},t=0 \right)
=& \frac{1}{\varphi(N)} \sum_{\eta \in \widehat{G(N)}_{1,-}} \overline{\eta}(r) L(0,\eta)^{-1}\\
&- \frac{1}{\varphi(N)} \sum_{\eta \in \widehat{G(N)}_{1,+}} \overline{\eta}(r) \frac{L''(0,\eta)}{2 L'(0,\eta)^2}
+ \frac{b_{\omega(N)}}{\varphi(N)}.
\end{split}
\]
\end{enumerate}
So
\begin{equation*}
\begin{split}
\mathscr{R}_{N,r} = 
&\frac{1}{\varphi(N)} \sum_{\substack{\eta \in \widehat{G(N)} \\ \eta(-1) = -1}} \overline{\eta}(r)  L(0,\eta)^{-1} - \frac{1}{\varphi(N)} \sum_{\substack{\eta \in \widehat{G(N)} \\ \eta(-1) = +1 \\ \eta \ne \eta_0}} \overline{\eta}(r) \frac{L''(0,\eta)}{2 L'(0,\eta)^2} \\
&+ \frac{1}{\varphi(N)} b_{\omega(N)},
\end{split}
\end{equation*}

\end{proof}

For any $n \in \mathbb{N}_{\geq 1}$ and positive integers $a$, $m$ with $\gcd(a,m) = 1$, we define
\begin{equation*}
\xi_n(s; m, a) := \sum_{\substack{k=1 \\ k^n \equiv\, a\, \mathrm{mod}{m}}}^\infty \frac{\Lambda(k)}{k^s}, \quad \mathrm{Re}(s) > 1,
\end{equation*}
where $\Lambda$ denotes the von Mangoldt function.\\

When $m = a = 1$ and $n = 1$,
\begin{equation*}
\xi_1(s; 1, 1) = -\frac{\zeta_\Q'(s)}{\zeta_\Q(s)}.
\end{equation*}

\begin{theorem}
The function $\xi_n(s; m, a)$ is analytic for $\mathrm{Re}(s) > 1$ and admits a meromorphic continuation to $\mathbb{C}$. Moreover,
\begin{equation*}
\mathrm{Res}\left( \frac{1}{s} \xi_n(s; m, a), s=0 \right)
= -\mathscr{Q}_n^+(m; a) - \mathscr{Q}_n^-(m; a) - \mathscr{Q}_n^0(m; a).
\end{equation*}
\end{theorem}

\begin{proof}

From Euler's product of $L(s,\chi)$, we get
\begin{equation*}
\log L(s,\chi^n) = \sum_{p\;\text{prime}} \sum_{\ell=1}^\infty \frac{\chi^n(p^\ell)}{\ell p^{\ell s}}.
\end{equation*}
Then
\begin{align*}
\sum_{\chi \in \widehat{G(m)}} \chi^{-1}(a) \log L(s,\chi^n)
&= \sum_{\ell=1}^\infty \sum_{p\;\text{prime}} \frac{ \sum_{\chi \in \widehat{G(m)}} \chi^{-1}(a) \chi(p^{n\ell}) }{ \ell p^{\ell s} } \\
&= \varphi(m) \sum_{\ell=1}^\infty \sum_{\substack{p\;\text{prime} \\ p^{n\ell} \equiv a\, \mathrm{mod}\, m}} \frac{1}{\ell p^{\ell s}}.
\end{align*}
Take the derivative, we obtain
\begin{equation*}
\xi_n(s; m, a) = -\frac{1}{\varphi(m)} \sum_{\chi \in \widehat{G(m)}} \chi^{-1}(a) \frac{L'(s,\chi^n)}{L(s,\chi^n)}.
\end{equation*}
We have
\begin{equation*}
\begin{split}
\mathrm{Res}\left( \frac{1}{s} \xi_n(s; m, a), s=0 \right)
= &-\mathscr{Q}_n^+(m; a) \\
&- \frac{1}{\varphi(m)} \sum_{\substack{\chi \in \widehat{G(m)} \\ \chi^n(-1) = -1}} \chi^{-1}(a) \frac{L'(0,\chi^n)}{L(0,\chi^n)} \\
&- \frac{1}{\varphi(m)} \left( \sum_{\substack{\chi \in \widehat{G(m)} \\ \chi^n = \chi_0}} \chi^{-1}(a) \right) \frac{L^{(\ell+1)}(0,\chi_0)}{(\ell+1) L^{(\ell)}(0,\chi_0)}.
\end{split}
\end{equation*}
Then follows from Lemma~\ref{formulas1}
\begin{equation*}
\begin{split}
\mathrm{Res}\left( \frac{1}{s} \xi_n(s; m, a), s=0 \right)
= &-\frac{1}{\varphi(m)} \sum_{\chi \in \widehat{G(m)}_{n,+} \setminus \{\chi_0\}} \chi^{-1}(a) \frac{L''(0,\chi^n)}{2 L'(0,\chi^n)} \\
&- \frac{1}{\varphi(m)} \sum_{\substack{\chi \in \widehat{G(m)} \\ \chi^n(-1) = -1}} \chi^{-1}(a) \frac{L'(0,\chi^n)}{L(0,\chi^n)} \\
&- \frac{1}{\varphi(m)} \left( \sum_{\substack{\chi \in \widehat{G(m)} \\ \chi^n = \chi_0}} \chi^{-1}(a) \right) \frac{L^{(\ell+1)}(0,\chi_0)}{(\ell+1) L^{(\ell)}(0,\chi_0)}.
\end{split}
\end{equation*}
This concludes the proof of the theorem.

\end{proof}

Let $a$ be a nonzero integer with $\gcd(a, m) = 1$. For $\mathrm{Re}(t) > 0$ and $\mathrm{Re}(s) > 0$, define
\begin{equation*}
\mathcal{P}_{m,a}(s; t) := \frac{1}{\varphi(m)} \sum_{n=1}^\infty \sum_{\chi \in \widehat{G(m)}} \frac{\mu(n)}{n^{t+1}} \chi^{-1}(a) \log L(ns, \chi^n).
\end{equation*}

\begin{proposition}

For $\mathrm{Re}(s) > 1$ and $t \in \mathbb{C}$,
\begin{equation*}
\frac{\pt \mathcal{P}_{m,a}}{\pt s}(s; t) = -\sum_{n=1}^\infty \frac{\mu(n)}{n^t} \xi_n(ns; m, a).
\end{equation*}
\end{proposition}

\begin{proof}

We have
\begin{equation*}
\exp(\mathcal{P}_\chi(s)) = \prod_{n=1}^\infty L(ns,\chi^n)^{\frac{\mu(n)}{n}}.
\end{equation*}
Thus,
\begin{equation*}
\mathcal{P}_\chi(s) = \sum_{n=1}^\infty \frac{\mu(n)}{n} \log L(ns,\chi^n).
\end{equation*}

It follows that
\begin{equation*}
\mathcal{P}_{m,a}(s) = \frac{1}{\varphi(m)} \sum_{\chi \in \widehat{G(m)}} \chi^{-1}(a) \mathcal{P}_\chi(s)
= \frac{1}{\varphi(m)} \sum_{n=1}^\infty \sum_{\chi \in \widehat{G(m)}} \frac{\mu(n)}{n} \chi^{-1}(a) \log L(ns,\chi^n).
\end{equation*}

Differentiating $\mathcal{P}_{m,a}(s; t)$ with respect to $s$, we obtain
\begin{align*}
\frac{1}{\varphi(m)} \sum_{n=1}^\infty \sum_{\chi \in \widehat{G(m)}} \frac{\mu(n)}{n^{t}} \chi^{-1}(a) \frac{L'(ns,\chi^n)}{L(ns,\chi^n)}
&= -\frac{1}{\varphi(m)} \sum_{\chi \in \widehat{G(m)}} \sum_{n=1}^\infty \frac{\mu(n)}{n^{t}} \chi^{-1}(a) \sum_{k=1}^\infty \frac{\Lambda(k) \chi(k^n)}{k^{ns}} \\
&= -\frac{1}{\varphi(m)} \sum_{n=1}^\infty \frac{\mu(n)}{n^t} \sum_{k=1}^\infty \frac{\Lambda(k)}{k^{ns}} \sum_{\chi \in \widehat{G(m)}} \chi^{-1}(a) \chi(k^n) \\
&= -\sum_{n=1}^\infty \frac{\mu(n)}{n^t} \sum_{\substack{k=1 \\ k^n \equiv a\, \mathrm{mod}\, m}}^\infty \frac{\Lambda(k)}{k^{ns}}.
\end{align*}

Therefore,
\[
\frac{\pt \mathcal{P}_{m,a}}{\pt s}(s; t) = -\sum_{n=1}^\infty \frac{\mu(n)}{n^t} \xi_n(ns; m, a).
\]

\end{proof}

\begin{proposition}\label{prop4.111} 
We have
\begin{equation*}
\begin{split}
\mathrm{Res}\left( \frac{1}{s} \frac{\pt \mathcal{P}_{m,a}}{\pt s}(s; t), s=0 \right)
= &\sum_{r=1}^{\varphi(m)} \mathscr{Q}_r^+(m; a) \sum_{\ell=0}^\infty \frac{\mu(\varphi(m)\ell + r)}{(\varphi(m)\ell + r)^t} \\
&+ \sum_{r=1}^{\varphi(m)} \mathscr{Q}_r^-(m; a) \sum_{\ell=0}^\infty \frac{\mu(\varphi(m)\ell + r)}{(\varphi(m)\ell + r)^t} \\
&+ \sum_{r=1}^{\varphi(m)} \mathscr{Q}_r^0(m; a) \sum_{\ell=0}^\infty \frac{\mu(\varphi(m)\ell + r)}{(\varphi(m)\ell + r)^t}.
\end{split}
\end{equation*}
\end{proposition}
\begin{proof}

Note that

\[
\mathrm{Res}\Bigl(\frac{1}{s}\sum_{n=1}^\infty \frac{\mu(n)}{n^t} 
\xi_n(ns;m,a),s=0 \Bigr)=\sum_{n=1}^\infty \frac{\mu(n)}{n^t} 
\mathrm{Res}\Bigl(\frac{1}{s}
\xi_n(ns;m,a),s=0 \Bigr).
\]

Then
\[
\begin{split}
\mathrm{Res}\Bigl(\frac{1}{s}\sum_{n=1}^\infty \frac{\mu(n)}{n^t} 
\xi_n(ns;m,a),s=0 \Bigr)=&\sum_{n=1}^\infty \frac{\mu(n)}{n^t}\mathscr{Q}_n^+(m;a) 
+\sum_{n=1}^\infty\frac{\mu(n)}{n^t} \mathscr{Q}_n^-(m;a)  \\
&+\sum_{n=1}^\infty \frac{\mu(n)}{n^t} \mathscr{Q}_n^0(m;a).
\end{split}
\]

From which we infer the following:
\[
\begin{split}
\mathrm{Res}\Bigl(\frac{1}{s}&\sum_{n=1}^\infty \frac{\mu(n)}{n^t} 
\xi_n(ns;m,a),s=0 \Bigr)=\sum_{r=1}^{\vf(m)} \mathscr{Q}_r^+(m;a)
  \sum_{\ell=0}^\infty \frac{\mu(\vf(m)l+r)}{(\vf(m)l+r)^t}  \\
&+\sum_{r=1}^{\vf(m)}  \mathscr{Q}_r^-(m;a)\sum_{\ell=0}^\infty \frac{\mu(\vf(m)l+r)}{(\vf(m)l+r)^t} +\sum_{r=1}^{\vf(m)} \mathscr{Q}_r^0(m;a)  \sum_{\ell=0}^\infty \frac{\mu(\vf(m)l+r)}{(\vf(m)l+r)^t}.
\end{split}
\]

This ends the proof of the proposition.
\end{proof}

\begin{theorem}\label{mainThm1}
We have
\begin{equation*}
\begin{split}
\mathrm{Res}\left( \frac{1}{t} \mathrm{Res}\left( \frac{1}{s} \frac{\pt \mathcal{P}_{m,a}}{\pt s}(s; t), s=0 \right), t=0 \right)
= \sum_{r=1}^{\varphi(m)} \left( \mathscr{Q}_r^+(m; a) + \mathscr{Q}_r^-(m; a) + \mathscr{Q}_r^0(m; a) \right) \mathscr{R}_{\varphi(m), r},
\end{split}
\end{equation*}
where $\mathscr{R}_{\varphi(m), r} = 0$ if $\gcd(\varphi(m), r)$ is not square-free.
\end{theorem}

\begin{proof}
From Proposition~\ref{prop4.111}, we have
\begin{equation*}
\mathrm{Res}\left( \frac{1}{t} \mathrm{Res}\left( \frac{1}{s} \frac{\pt \mathcal{P}_{m,a}}{\pt s}(s; t), s=0 \right), t=0 \right)
= \mathrm{Res}\left( \frac{1}{t} \mathrm{Res}\left( \frac{1}{s} \sum_{n=1}^\infty \frac{\mu(n)}{n^t} \xi_n(ns; m, a), s=0 \right), t=0 \right).
\end{equation*}

Thus,
\begin{equation*}
\begin{split}
\mathrm{Res}\left( \frac{1}{t} \mathrm{Res}\left( \frac{1}{s} \sum_{n=1}^\infty \frac{\mu(n)}{n^t} \xi_n(ns; m, a), s=0 \right), t=0 \right)
= &\sum_{r=1}^{\varphi(m)} \mathscr{Q}_r^+(m; a) \mathscr{R}_{\varphi(m), r} \\
&+ \sum_{r=1}^{\varphi(m)} \mathscr{Q}_r^-(m; a) \mathscr{R}_{\varphi(m), r} \\
&+ \sum_{r=1}^{\varphi(m)} \mathscr{Q}_r^0(m; a) \mathscr{R}_{\varphi(m), r}.
\end{split}
\end{equation*}

Theorem~\ref{mainThm1} follows immediately.

\end{proof}

\begin{theorem}\label{pam}
\[
\Rprod_{\substack{p\ \text{prime} \\ p\equiv a\, \mathrm{mod}\, m}} p = 
\exp\left( \sum_{r=1}^{\varphi(m)} \left( \mathscr{Q}_r^+(m; a) + \mathscr{Q}_r^-(m; a) + \mathscr{Q}_r^0(m; a) \right) \mathscr{R}_{\varphi(m), r} \right).
\]
\end{theorem}

\begin{proof}
From the definition of the super-regularized product and the above computations, we have:
\begin{equation*}
\begin{split}
\Rprod_{\substack{p\ \text{prime} \\ p\equiv a\, \mathrm{mod}\, m}} p
&= \exp\left(
\mathrm{Res}\left(
-\frac{1}{t} \mathrm{ext}_t \left(
\mathrm{Res}\left(
\frac{1}{s} \mathrm{ext}_s \frac{\partial \mathcal{P}_{m,a}}{\partial s}(s,t), s=0
\right)
\right), t=0
\right)
\right).
\end{split}
\end{equation*}

By Theorem~\ref{mainThm1}, it follows that
\begin{equation*}
\begin{split}
\Rprod_{\substack{p\ \text{prime} \\ p\equiv a\, \mathrm{mod}\, m}} p
&= \exp\left(
\mathrm{Res}\left(
\frac{1}{t} \mathrm{ext}_t \left(
\mathrm{Res}\left(
\frac{1}{s} \mathrm{ext}_s \sum_{n=1}^\infty \frac{\mu(n)}{n^t} \xi_n(ns; m, a), s=0
\right)
\right), t=0
\right)
\right) \\
&= \exp\left(
\sum_{r=1}^{\varphi(m)} \left( \mathscr{Q}_r^+(m; a) + \mathscr{Q}_r^-(m; a) + \mathscr{Q}_r^0(m; a) \right) \mathscr{R}_{\varphi(m), r}
\right).
\end{split}
\end{equation*}
\end{proof}

\subsection{The case $m=4$}

We recall the values of the Dirichlet characters $\chi \pmod{4}$. They are summarized in the following table:
\bigskip 

\begin{center}
\begin{tabular}{c|c|c}
\hline
$n \pmod{4}$ & $1$ & $3$ \\
\hline
$\chi_1(n)$ & $1$ & $1$ \\
$\chi_2(n)$ & $1$ & $-1$ \\
\hline
\end{tabular}
\end{center}
\bigskip

The Dirichlet $L$-function for the nontrivial character $\chi_2 \, (\mathrm{mod}\, 4)$ is
\[
L(s, \chi_2) = \frac{1}{4^s} \zeta_H\!\left(s, \frac{1}{4}\right) - \frac{1}{4^s} \zeta_H\!\left(s, \frac{3}{4}\right).
\]

At $s=0$, we compute:
\[
L(0, \chi_2) = -B_{1,\chi_2} = -\frac{1}{4} \left( \chi_2(1) + \chi_2(3) \cdot 3 \right) = \frac{1}{2}.
\]

The derivative at $s=0$ is given by:
\[
L'(0, \chi_2) = \log 2 + \log \frac{\Gamma\!\left(\frac{3}{4}\right)}{\Gamma\!\left(\frac{1}{4}\right)}.
\]

From \eqref{8},
\[
 \frac{L^{''} (0,{\chi_0})}{ 2L^{'} (0,{\chi_0})}=\frac{\zeta_\Q'(0)}{\zeta_\Q(0)}-\frac{1}{2}\log 2.
 \]

An easy computation shows  that
\[
\mathscr{R}_{2,1}=\mathrm{Res}\Bigl(\frac{1}{t}\sum_{\ell=0}^\infty \frac{\mu(2l+1)}{(2l+1)^t},t=0\Bigr)=-1-\frac{4}{\log 2} \zeta_\Q'(0).
\]
So
\[
\begin{split}
\mathscr{R}_{2,2}=&\mathrm{Res}\Bigl(\frac{1}{t}\sum_{\ell=1}^\infty    \frac{\mu(2l)}{(2l)^t},t=0 \Bigr)\\
=&-\mathrm{Res}\Bigl(\frac{1}{t}\sum_{\ell=0}^\infty  \frac{1}{2^t}  \frac{\mu(2l+1)}{(2l+1)^t},t=0\Bigr)\\ 
=&-1+\frac{4}{\log 2} \zeta_\Q'(0).
\end{split}
\]

It follows that
\begin{equation*}
\begin{split}
\sum_{r=1}^{2} \Bigl( \mathscr{Q}_r^+(4; a) + \mathscr{Q}_r^-(4; a) +& \mathscr{Q}_r^0(4; a) \Bigr) \mathscr{R}_{2,r}\\
= &\left( \mathscr{Q}_1^-(4; a) + \mathscr{Q}_1^0(4; a) \right) \mathscr{R}_{2,1} + \delta_{1,a} \mathscr{Q}_2^0(4; a) \mathscr{R}_{2,2} \\
= &-\chi_2^{-1}(a) \log \left( 2 \frac{\Gamma\!\left(\frac{3}{4}\right)}{\Gamma\!\left(\frac{1}{4}\right)} \right) \left( \frac{2}{\log 2} \log (2\pi) - 1 \right) \\
&+ \log \left( \frac{2\pi}{\sqrt{2}} \right) \left( \delta_{1,a} \left( 1 + \frac{2}{\log 2} \log 2\pi \right) - \frac{\log 2\pi}{\log 2} + \frac{1}{2} \right),
\end{split}
\end{equation*}
where $\delta_{1,a} = 1$ if $a \equiv 1 \pmod{4}$, and $\delta_{1,a} = 0$ if $a \equiv 3 \pmod{4}$.\\

Hence
 \[
 \begin{split}
-\frac{\pt}{\pt s} \Bigl( \sum_{p\equiv a \ (\text{mod} \ 4)} \frac{1}{p^s} \Bigr)_{|_{s=0}}
=& -\chi_2^{-1}(a)\log \left( 2\frac{\Gamma(\frac{3}{4})}{\Gamma(\frac{1}{4})}\right)\left(\frac{2}{\log 2}\log (2\pi)-1 \right) \\
 &+\log \left( \frac{2\pi}{\sqrt{2}}\right) \left( \delta_{1,a} \left( 1+\frac{2}{\log 2} \log 2\pi\right) -\frac{\log 2\pi}{\log 2}+\frac{1}{2} \right)
 \end{split}
 \]
 Recall the following classical identities: 
\[
\Gamma(\frac{1}{4}) \Gamma(\frac{3}{4})=\sqrt{2}\pi \quad \text{and}\quad G=\frac{\Gamma(\frac{1}{4})^2}{ 2\sqrt{2\pi^3}},
\]
where $G$ is Gauss's constant. Hence, we obtain
\begin{proposition}\label{m=4}

  \[
 \begin{split}
 \Rprod_{p \equiv 3\ (\mathrm{mod}\,{4})} p
=& \exp\left(-\log \left( {\sqrt{\pi}G} \right)\left(\frac{2}{\log 2}\log (2\pi)-1 \right) 
 +\log \left( \frac{2\pi}{\sqrt{2}}\right) \left(  -\frac{\log 2\pi}{\log 2}+\frac{1}{2} \right)\right),
 \end{split}
 \]
 and
 \[
 \begin{split}
 \Rprod_{p \equiv 1 \, (\mathrm{mod}\, 4)} p
=& \exp\left(\log({\sqrt{\pi}G})\left(\frac{2}{\log 2}\log (2\pi)-1 \right) 
 +\log \left( \frac{2\pi}{\sqrt{2}}\right) \left(   \frac{3}{2}+\frac{1}{\log 2} \log 2\pi \right)\right).
 \end{split}
 \]

\end{proposition}

\begin{remark} 

\item 

\begin{enumerate}

\item  Numerical computations show that the values in Proposition~\ref{m=4} are not integers. We conclude that there are infinitely many primes in the  {arithmetic}  progression $1 \bmod 4$ (respectively, $3 \bmod 4$). 

\item 
We have
\begin{equation*}
\begin{split}
\Rprod_{p \equiv 1 \, (\mathrm{mod}\, 4)} p \cdot \Rprod_{p \equiv 3 \, (\mathrm{mod}\, 4)} p
&= \exp\left( -\log 2 - 4 \zeta_\Q'(0) \right) \\
&= 2\pi^2.
\end{split}
\end{equation*}

On the other hand, we know from \cite{Marco_prime} that
\[
\Rprod_{p} p = 4\pi^2.
\]
This confirms our earlier computations: the super-regularized product over all primes decomposes as expected into the product over residue classes modulo 4, and the resulting value $2\pi^2$ is consistent with the known global result $\Rprod_p p = 4\pi^2$, since
\[
\Rprod_{p \equiv 1 \, (\mathrm{mod}\, 4)} p \cdot \Rprod_{p \equiv 3 \, (\mathrm{mod}\, 4)} p = 2\pi^2 = \frac{1}{2} \cdot 4\pi^2.
\]

\end{enumerate}
\end{remark}


\section{Zeta functions of schemes over finite fields}\label{schemefinite}

In this section, we consider the class of zeta functions associated with schemes defined over finite fields. More precisely, let $X$ be a smooth projective scheme over a finite field $\mathbb{F}_q$. For each positive integer $m$, denote by $N_m$ the number of $\mathbb{F}_{q^m}$-rational points of $X$, i.e., the cardinal  of $X(\mathbb{F}_{q^m})$. The {zeta function} of $X$ is then defined as the formal power series
\[
Z(X, t) = \exp\left( \sum_{m=1}^\infty N_m \frac{t^m}{m} \right) \in \mathbb{Q}[[t]],
\]
which converges formally in the ring of power series with rational coefficients. Note that its logarithmic derivative satisfies
\[
\frac{d}{dt} \log Z(X, t) = \sum_{m \geq 1} N_m t^{m-1}.
\]
 $Z(X,t)$ satisfies the functional equation of the form
\[
Z\left(X, \frac{1}{q^{\dim X} t}\right) = \pm q^{e\dim X /2} t^e Z(X,t)
\]
for some integer $e$.\\

The scheme $X$ may also be viewed as a scheme of finite type over $\mathrm{Spec} \mathbb{Z}$ via the composition $
X \to \mathrm{Spec}\,\mathbb{F}_q \to \mathrm{Spec}\,\mathbb{Z}.$

One can show that the zeta function admits an Euler product expansion:
\[
Z(X, t) = \prod_{x} \frac{1}{1 - t^{\deg x}},
\]
where the product runs over all closed points $x$ of $X$, and $\deg x = [\kappa(x) : \mathbb{F}_q]$ denotes the degree of the residue field extension. Consequently, the {Hasse-Weil zeta function} of $X$ is given by
\[
\zeta(X, s) = Z(X, q^{-s}).
\]

This function has poles along the vertical lines $\mathrm{Re}(s) = 0, 1, 2, \ldots, \dim X$, and zeros along the lines $\mathrm{Re}(s) = \frac{1}{2}, \frac{3}{2}, \ldots, \frac{\dim X - 1}{2}$. Moreover, there exist polynomials $P_i(X, t) = 1 + \cdots \in \mathbb{Z}[t]$, for $i = 0, \ldots, 2\dim X$, such that
\[
\zeta(X, s) = \frac{P_1(X, q^{-s}) \cdots P_{2\dim X-1}(X, q^{-s})}{P_0(X, q^{-s}) P_2(X, q^{-s}) \cdots P_{2\dim X}(X, q^{-s})},
\]

 In particular, for any integer $r$, we have the local behavior
\begin{equation}\label{near0}
\zeta(X, s) \sim \frac{C_X(r)}{(1 - q^{r-s})^{\rho_r}} \quad \text{as } s \to r,
\end{equation}
for some integer $\rho_r$ and rational number $C_X(r)$, where ``$\sim$'' means that
\[
\lim_{s \to r} (1 - q^{r-s})^{\rho_r} \zeta(X, s) = C_X(r).
\]
For further details, see \cite{DixExposes, Milne1, Tate-zeta1, Tate-abelian, Tate-zeta2, Schneider-zeta}.

In the case where $X = C$ is a curve, the situation is slightly more tractable. Below, we provide concrete examples in which the Hasse--Weil zeta function can be computed explicitly.
\begin{example}
\item 
\begin{enumerate}

\item

The Fermat cubic curve $\mathcal{F}_3$ is given by the equation
\[
x^3 + y^3 + z^3 = 0
\]
over $\mathbb{F}_2$.  The  zeta function $Z(\mathcal{F}_3, t)$ is given by
\[
Z(\mathcal{F}_3, t) = \frac{L(t)}{(1 - t)(1 - 2t)},
\]
where
\[
L(t) = 1 + 2t^2 = (1 - i\sqrt{2}t)(1 + i\sqrt{2}t).
\]

\item 
Let $\mu_N$ denote the group of $N$-th roots of unity in $\mathbb{C}$.  Let $p$ be a prime number. Let $\chi_1, \chi_2 : \mathbb{F}_p^* \to \mu_{p-1}$ be characters with $\chi_1(0)=\chi_2(0) = 0$. The Jacobi sum is defined by
\[
J(\chi_1, \chi_2) = \sum_{x \in \mathbb{F}_p} \chi_1(x) \chi_2(1 - x).
\]
 
Let $\ell$ be a odd prime number.  
Let $C$ be the Fermat curve defined by
\[
x^\ell + y^\ell + z^\ell = 0
\]
over $\mathbb{F}_p$. Weil  showed that the zeta function $Z(C, t)$ of $C$ is given by

\[
Z(C, t) = \frac{
    \displaystyle\prod_{\substack{a, b \in \mathbb{F}_\ell \\ ab(a + b) \not\equiv 0 \pmod{\ell}}}
    (1 - \alpha_{a,b} t)
}{
    (1 - t)(1 - pt)
},
\]

where $
\alpha_{a,b} = \chi^{(a+b)}(-1) J(\chi^a, \chi^b),
$
for $a, b \in \mathbb{F}_\ell$ such that $ab(a + b) \not\equiv 0 \mod \ell$.

\item The Klein quartic curve
is a 
 curve in the projective space $\mathbb{P}^2$ defined  over $\mathbb{F}_2$ defined by the following  equation:
\[
x^3 y + y^3 z + z^3 x = 0.
\]
The zeta function in this case is
\[
Z(C, t) = \frac{L(t)}{(1 - t)(1 - 2t)},
\]
where $
L(t) = 1 + 5t^3 + 8t^6.$
\item  Let $C$ be the hyperelliptic curve $y^2 + y = x^5 + 1$ over $\mathbb{F}_2$.
The  zeta function of this curve is
\[
Z(C,t) = \frac{L(t)}{(1-t)(1-2t)},
\]
where $L(t)=(1+2t-2t^2)(1+2t+2t^2)$.
\end{enumerate}
\end{example}


\begin{definition}
Let $X$ be a smooth projective scheme over a finite field $\mathbb{F}_q$. We define 
\[
\mathscr{P}_X(t) = \sum_{x \in X^0} t^{\deg(x)},
\]
where the sum runs over the closed points of $X$, and $\deg(x)$ denotes the degree of the closed point $x$. And, we set
\[
\mathcal{P}_X(s):=\mathscr{P}_X(q^{-s}).
\]
and
\[
\mathcal{P}_X(s,u):=\sum_{n=1}^\infty \frac{1}{n^{u+1}} \log \zeta(X,ns).
\]

\end{definition}

\begin{proposition}
Let $X$ be a smooth projective scheme over a finite field $\mathbb{F}_q$. The series $\mathscr{P}_X(t)$ converges absolutely for $|t| < 1$. Moreover, $\mathcal{P}_X(s,u)$ converges for $\mathrm{Re}(s) > 1$ and $u \in \mathbb{C}$, and it can be extended to a meromorphic function around $(0,0)$, with
\[
\mathcal{P}_X(s,0) = \mathcal{P}_X(s).
\]
Furthermore,
\[
\mathscr{P}_X'(t) = \sum_{n=1}^\infty \mu(n) \cdot \frac{Z'(X,t^n)}{Z(X,t^n)} \cdot t^{n-1},
\]
for all $|t| < 1$, where $\mu(n)$ is the M\"obius function.
\end{proposition}

\begin{proof} The series $\mathscr{P}_X(t)$ converges absolutely for $|t| < 1$, as it appears in the logarithmic derivative of the zeta function $Z(X,t)$, which is known to converge absolutely in this region.  We infer that $\mathcal P_X(s,u)$ converges for 
$\mathrm{Re}(s)>1$ and $u\in \C$.

 We have 
\[
\exp(\mathscr P_X(t)) = \prod_n \prod_{x \in X^0} (1 - t^{n \deg(x)})^{-\mu(n)/n}
\]
\[
= \prod_n Z(X,t^n)^{\mu(n)/n}.
\]
It follows that

\[ 
\mathcal P_X(s)=\mathcal{P}_X(s,0),\ \text{and}\quad 
\mathscr P_X'(t) = \sum_n \mu(n) 
\frac{Z'(X,t^n)}{Z(X,t^n)} t^{n-1}.
\]
\end{proof}


From \eqref{near0}, we may write
\[
\zeta(X,s) = \frac{g(s)}{s^\rho},
\]
where $g(s)$ is analytic in a neighborhood of $s=0$ and $\rho = \rho_0$ is an integer. Consequently,
\[
g(s) = s^\rho \zeta(X,s) = s^\rho Z(X,q^{-s}).
\]
Let us write
\[
\zeta(X,s) = \frac{g(0)}{s^\rho} + \frac{g'(0)}{s^{\rho-1}} + O\left(\frac{1}{s^{\rho-2}}\right) \quad \text{as } s \to 0.
\]

Our next goal is to show that Proposition~\ref{PropF} and Theorem~\ref{ThmF}, which we prove below, provide a refinement of \eqref{near0}. For simplicity, we introduce the function $\widetilde{F}(t)$, defined implicitly by the relation
\[
g(s) = s^\rho  \frac{\widetilde{F}(q^{-s})}{(1 - q^{-s})^\rho}, \quad \text{where } \rho = \rho_0.
\]
By the definition of $\rho_0$, it follows that $\widetilde{F}(1) \ne 0$.

\begin{proposition}\label{PropF}
Under the above assumptions, we have
\[
g(0) = \frac{C_X(0)}{(\log q)^{\rho_0}}\quad \text{and}\quad
\widetilde{F}(1) = C_X(0).
\]
and

\[
\frac{g'(0)}{g(0)} 
= \frac{1}{(\log q)^{\rho-1}} \left( \frac{\rho}{2} - \frac{ \widetilde{F}'(1) }{ C_X(0) } \right).
\]
\end{proposition}

\begin{proof}
For small $ s $, we have 
\[
q^{-s} = 1 - (\log q)s + \frac{(\log q)^2 s^2}{2} - \cdots
\]

Thus, for sufficiently small $s$, 
\begin{align*}
g(s) 
&= \frac{\widetilde{F}(q^{-s})}{(\log q)^\rho \left(1 - \frac{\log q}{2}s + \cdots \right)^\rho}.
\end{align*}

 From  \eqref{near0}, it follows easily that
\[
\widetilde{F}(1) = C_X(0).
\]

Observe that $ \dfrac{s^\rho}{(1 - q^{-s})^\rho} \to \dfrac{1}{(\log q)^\rho} $ as $ s \to 0 $. Therefore, taking the limit as $ s \to 0 $, we obtain:
\[
g(0) = \lim_{s \to 0} g(s) = \frac{\widetilde{F}(1)}{(\log q)^\rho} = \frac{C_X(0)}{(\log q)^\rho}.
\]

Now, we  compute the derivative $ g'(s) $:
\[
\begin{split}
g'(s) = 
&\frac{ -(\log q) q^{-s} \widetilde{F}'(q^{-s}) }{ (\log q)^\rho \left(1 - \frac{\log q}{2}s + \cdots \right)^\rho } \\
&- \frac{\rho \widetilde{F}(q^{-s}) }{ (\log q)^\rho \left(1 - \cdots \right)^{\rho+1} } \cdot \left( -\frac{\log q}{2} + o(1) \right).
\end{split}
\]

Evaluating at $ s = 0 $:
\[
g'(0) = -(\log q) \frac{ \widetilde{F}'(1) }{ (\log q)^\rho } + \frac{\rho \log q}{2 (\log q)^\rho} \widetilde{F}(1).
\]

Dividing by $ g(0) = \dfrac{C_X(0)}{(\log q)^\rho} $, we obtain the logarithmic derivative:
\[
\begin{split}
\frac{g'(0)}{g(0)} 
&= -\frac{ \widetilde{F}'(1) }{ (\log q)^{\rho-1} \widetilde{F}(1) } + \frac{\rho}{2 (\log q)^{\rho-1}} \\
&= \frac{1}{(\log q)^{\rho-1}} \left( \frac{\rho}{2} - \frac{ \widetilde{F}'(1) }{ C_X(0) } \right).
\end{split}
\]

\end{proof}

\begin{theorem}\label{ThmF}
Let $X$ be a smooth projective scheme over a finite field $\mathbb{F}_q$. We have

\[
\Rprod_{x \in X^0} N (x)=\exp\left(  1 - 2\frac{\widetilde{F}'(1)}{C_X(0)}  \right)
\]

\end{theorem}

\begin{proof}
We have
\[
\begin{split}
\frac{\partial }{\partial s} \mathcal P_X(s,u)=&\sum_{n=1}^\infty \frac{1}{n^u}\frac{\zeta'(X,ns)}{ \zeta(X,ns)}\\
 =& \sum_{n=1}^\infty \frac{1}{n^u}\frac{g'(ns)}{ g(ns)}.
\end{split}
\]
Hence 
\[
\mathrm{Res}\left(\frac{1}{u}\text{ext}_u \left(\mathrm{ext}_s\frac{1}{s}\frac{\partial }{\partial s} \mathcal{P}_X(s,u))  , s=0 \right), u=0 \right)=\frac{1}{\zeta_\Q(0)}\frac{g'(0)}{g(0)}
\]

It is enough to see that
\[
\mathcal P_X'(0) = \frac{1}{\zeta_\Q(0)} \frac{g'(0)}{g(0)} = \frac{-2}{(\log q)^{\rho-1}} \left( \frac{\rho}{2} - \frac{\widetilde{F}'(1)}{C_X(0)} \right),
\]
and that $\rho=1$.

\end{proof}

\begin{example}
Let $\mathcal F_3$ be the Fermat cubic curve giver by the equation $x^3+y^3+z^3=0$ over $\mathbb{F}_2$. We show easily that

\[
\Rprod_{x \in \mathcal{F}_3^0} \deg(x) =\exp\left(\frac{7}{3}\right).
\]
\end{example}

\subsection{Regularized product of primes over global function fields}\label{RFF}

Let us restrict our attention to global function fields in one variable with finite constant field $\mathbb{F}_q$. We briefly recall some basic facts about their zeta functions.

Let $K$ be a global function field of genus $g$. For details and proofs of the results below, see \cite{Rosen}. The zeta function $\zeta_K(s)$ is given by:
\[
\zeta_K(s) = \frac{L_K(q^{-s})}{(1 - q^{-s})(1 - q^{1-s})},
\]
where $L_K(u) = \prod_{i=1}^{2g} (1 - \pi_i u)$ is a polynomial with integer coefficients, and the roots $\pi_1, \dots, \pi_{2g}$ are complex numbers of absolute value $\sqrt{q}$, see \cite[Theorem~5.9]{Rosen}.\\

Taking the logarithmic derivative yields
\[
\frac{L_K'(u)}{L_K(u)} = \sum_{i=1}^{2g} \frac{-\pi_i}{1 - \pi_i u}.
\]

Hence, using the above notation, we obtain
\[
\frac{\widetilde{F}'(1)}{C_X(0)} = \frac{L_K'(1)}{L_K(1)} + \frac{q  L_K(1)}{1 - q}.
\]

This simplifies further to
\[
\frac{\widetilde{F}'(1)}{C_X(0)} = \sum_{i=1}^{2g} \frac{\pi_i}{\pi_i - 1} + \frac{q \cdot h_K}{1 - q},
\]
where $h_K$ denotes the class number of $K$.\\

So Theorem~\ref{ThmF} takes the following form in the case of function fields.

\begin{theorem}
Let $K$ be a global function field of genus $g$. Then the regularized product over all nonzero prime ideals $\mathfrak{p}$ of $K$ satisfies:
\[
\Rprod_{\mathfrak{p}} N(\mathfrak{p}) = \exp\left( 2 \left( \frac{1}{2} + \sum_{i=1}^{2g} \frac{1}{1 - \pi_i} - 2g + \frac{q h_K}{q - 1} \right) \right),
\]
where $N(\mathfrak{p})$ denotes the norm of $\mathfrak{p}$, and $h_K$ is the class number of $K$.
\end{theorem}

\begin{example} Let us consider $K = \mathbb{F}_q(T)$, the field of rational fractions with coefficients in $\mathbb F_q$. 
We know that
\[
\zeta_{\mathbb F_q(T)}(s) = \frac{1}{(1 - q^{-s})(1 - q^{1-s})}.
\]
The regularized product over all nonzero prime ideals $\mathfrak{p}$ of $\mathbb F_q[T]$ satisfies:
\[
\Rprod_{\mathfrak{p}} N(\mathfrak{p}) = \exp\left( \frac{3q-1}{ q-1} \right).
\]
Note that $\exp\left( \frac{3q-1}{ q-1} \right)$ is a irrational number. We infer, analogsly to the main application  in \cite{Marco_prime}, that the number of irreducible polynomials of $\mathbb F_q[T]$ is infinite. \\

\end{example}

\subsection{Regularized product of primes in an arithmetic progression}\label{RFFP}

Let $\mathbb F$ be finite field with $q$ elements. Let $A$ be the ring of polynomials in one variable with coefficients in $\mathbb F$.\\

Let $m$ be an element of $A$ of positive degree.   Let $\chi$ be a Dirichlet character modulo $m$. The Dirichlet $L$-series corresponding to $\chi$ is defined by
\[
L(s, \chi) = \sum_{f \text{ monic}} \frac{\chi(f)}{|f|^s},
\]
where $|f|=q^{\deg f}$. \\

 Let $\chi$ be a non-trivial Dirichlet character modulo $m$. Then, $L(s, \chi)$ is a polynomial in $q^{-s}$ of degree $M-1$ with   $M\leq \deg(m)$, see \cite[Propositon 4.3]{Rosen}.    We have
\[
L^*(u, \chi) = \sum_{k=0}^{M-1} a_k(\chi) u^k = \prod_{i=1}^{M-1} (1 - \alpha_i(\chi) u) \,.
\]
where we have set $L^\ast(
q^{-s},\chi):=L(s,\chi)$. The analogue of the Riemann hypothesis for function fields over a finite field  states that each of the roots $\alpha_k(\chi)$ has absolute value either $1$ or $\sqrt{q}$.  \\

Next, we recall some  properties that are essentially consequences of a result of Weil 
\cite{Weil_1945}, see also \cite{Rosen}. Let $\chi$ be non-principal Dirichlet character $\chi$ modulo $m$. We have
 \[
    L^\ast(u, \chi) = L^\ast(u, \chi^*) \prod_{\substack{P \mid m \\ P \nmid m(\chi^*)}} (1 - u^{\deg(P)}),
    \]
    where $\chi^*$ is the primitive Dirichlet character modulo a polynomial $m(\chi^*)$ which induces $\chi$. The polynomial $L^\ast(u, \chi^*)$ is of  degree $M(\chi^*) - 1$; and 

\[
\begin{cases}
\quad L^\ast(u, \chi^*) = (1 - u) \prod_{i=1}^{M(\chi^*) - 2} (1 - \alpha_i(\chi^*) u),\quad \text{If } \chi^* \text{ is even},  \\
 \\
\quad L^\ast(u, \chi^*) = \prod_{i=1}^{M(\chi^*) - 1} (1 - \alpha_i(\chi^*) u)\quad\quad\quad\quad \quad \text{otherwise}.
\end{cases}
\]
\bigskip

From the above, it follows that a simpler expression for
\[
\mathrm{Res}\left( \frac{1}{s}  \frac{L'(s,\chi)}{L(s,\chi)}, s = 0 \right)
\]
can be obtained, in contrast to the case over $\mathbb{Z}$. Thus, we obtain a more explicit formula for
\[
\Rprod_{\substack{P \,\text{monic irreducible}  \\
P\equiv\, a \, \mathrm{mod}\, m }}|P|,
\]
where we have omitted the proof of Theorem~\ref{pam} in the function field setting, as it is analogous to the case of $\Z$.

\bibliographystyle{plain}

\bibliography{biblio}

\end{document}